\documentclass{agtart_a}
\pdfoutput=1
\usepackage{pinlabel}

\title{Generating family invariants for Legendrian links of unknots}

\author{Jill Jordan}
\givenname{Jill}
\surname{Jordan}
\address{Department of Mathematics\\
Bryn Mawr College\\\newline
Bryn Mawr PA 19010\\
USA}
\email{jill.e.jordan@hotmail.com}
\urladdr{}

\author{Lisa Traynor}
\givenname{Lisa}
\surname{Traynor}
\address{Department of Mathematics\\
Bryn Mawr College\\\newline
Bryn Mawr PA 19010\\
USA}
\email{ltraynor@brynmawr.edu}
\urladdr{}

\volumenumber{6}
\issuenumber{}
\publicationyear{2006}
\papernumber{34}
\startpage{895}
\endpage{933}

\doi{}
\MR{}
\Zbl{}


\received{28 March 2006}
\revised{}
\accepted{25 April 2006}
\published{24 July 2006}
\publishedonline{24 July 2006}
\proposed{}
\seconded{}
\corresponding{}
\editor{}
\version{}

\arxivreference{}

\let\xysavmatrix\xymatrix
\def\xymatrix{\disablesubscriptcorrection\xysavmatrix}
\AtBeginDocument{\let\bar\wbar\let\tilde\wtilde\let\hat\what}

\makeatletter
\def\cnewtheorem#1[#2]#3{\newtheorem{#1}{#3}[section]
\expandafter\let\csname c@#1\endcsname\c@theorem}

\newtheorem{theorem}{Theorem}[section]
\cnewtheorem{proposition}[theorem]{Proposition}
\cnewtheorem{lemma}[theorem]{Lemma}
\cnewtheorem{corollary}[theorem]{Corollary}
\theoremstyle{definition}
\cnewtheorem{definition}[theorem]{Definition}
\theoremstyle{remark}
\cnewtheorem{remark}[theorem]{Remark}

\makeatother

\def\jetS{\mathcal{J}^1(S^1)}
\def\jetR{\mathcal{J}^1(\R)}
\def\jetRm{\mathcal{J}^1(\R^m)}
\def\jetM{\mathcal{J}^1(M)}

\numberwithin{theorem}{section}

\begin{document}

\begin{asciiabstract}
Theory is developed for linear-quadratic at infinity generating families
for Legendrian knots in R^3.   It is shown that the unknot
with maximal Thurston--Bennequin invariant of -1 has a unique
linear-quadratic at infinity generating family, up to fiber-preserving
diffeomorphism and stabilization.  From this, invariant generating
family polynomials are constructed for 2-component Legendrian links
where each component is a maximal unknot.  Techniques are developed to
compute these polynomials, and computations are done for two families
of Legendrian links:  rational links and twist links.  The polynomials
allow one to show that some topologically equivalent links with the same
classical invariants are not Legendrian equivalent.  It is also shown
that for these families of links the generating family polynomials agree
with the polynomials arising from a linearization of the differential
graded algebra associated to the links.
\end{asciiabstract}

\begin{htmlabstract}
Theory is developed for linear-quadratic at infinity generating families
for Legendrian knots in <b>R</b><sup>3</sup>.   It is shown that the unknot
with maximal Thurston&ndash;Bennequin invariant of  -1 has a unique
linear-quadratic at infinity generating family, up to fiber-preserving
diffeomorphism and stabilization.  From this, invariant generating
family polynomials are constructed for 2&ndash;component Legendrian links
where each component is a maximal unknot.  Techniques are developed to
compute these polynomials, and computations are done for two families
of Legendrian links:   rational links and twist links.   The polynomials
allow one to show that some topologically equivalent links with the same
classical invariants are not Legendrian equivalent.  It is also shown
that for these families of links the generating family polynomials agree
with the polynomials arising from a linearization of the differential
graded algebra associated to the links.
\end{htmlabstract}

\begin{abstract}
Theory is developed for linear-quadratic at infinity generating families
for Legendrian knots in $\mathbb{R}^3$.   It is shown that the unknot
with maximal Thurston--Bennequin invariant of  $-1$ has a unique
linear-quadratic at infinity generating family, up to fiber-preserving
diffeomorphism and stabilization.  From this, invariant generating
family polynomials are constructed for $2$--component Legendrian links
where each component is a maximal unknot.  Techniques are developed to
compute these polynomials, and computations are done for two families
of Legendrian links:   rational links and twist links.   The polynomials
allow one to show that some topologically equivalent links with the same
classical invariants are not Legendrian equivalent.  It is also shown
that for these families of links the generating family polynomials agree
with the polynomials arising from a linearization of the differential
graded algebra associated to the links.
\end{abstract}

\maketitle

\section{Introduction} \label{intro}

A basic problem in contact topology is determining when two Legendrian knots or links
are equivalent.  Two Legendrian links that are topologically equivalent can 
sometimes be distinguished via the classical Legendrian invariants of
the rotation and Thurston--Bennequin numbers for the components.  \ In recent years, new invariants for
Legendrian  links have come from Legendrian contact homology,
(see Chekanov \cite{Chekanov}, Etnyre, Ng and Sabloff \cite{ENS}, and
Eliashberg, Givental and Hofer \cite{EGH}).  One can use the theory of
holomorphic curves to associate a differential, graded algebra (DGA) to
a Legendrian link. \ It is sometimes possible to associate an invariant
polynomial to a Legendrian link
 by means of this DGA (see, for example, \cite{Chekanov}, Ng \cite{Ng1},
and Ng and Traynor \cite{Ng-Traynor}).

New invariants for some Legendrian links have also
come from the theory of generating families.  The theory of generating families, also known as generating functions, is quite classic; some history can be found, for example,
in Eliashberg and Gromov \cite{EG}.  In the early 1990's,
the technique of generating functions received renewed attention due to the
work of Viterbo \cite{Viterbo}.  Viterbo found that the $0$--section of a cotangent bundle
has a unique quadratic at infinity generating family. A careful proof of
this uniqueness statement was carried out by David Th\'{e}ret in
\cite{Theret}. This uniqueness result has many interesting symplectic
applications; see, for example,
\cite{Viterbo}, The\'ret \cite{TheretCamel}, and Traynor \cite{Symplectic}.
By some standard identifications, this
uniqueness result leads to a uniqueness result for quadratic at infinity
generating functions of the Legendrian $1$--jet of the zero function in
$\mathcal{J}^{1}\left( M\right) $.  From this uniqueness result in the
contact setting, it is possible to construct  invariant
polynomials for two component links in $\mathcal{J}^{1}\left( S^{1}\right) $
when each component is Legendrian isotopic to the $1$--jet of a function
$f\co S^{1}\rightarrow \R$ (see Traynor \cite{Generating}).

The particular focus of this paper is to associate a pair of invariant
polynomials to a specific type of two-component Legendrian link, one in
which each component is a Legendrian unknot with maximal Thurston--Bennequin
number of $-1$.
We will call such an unknot a \emph{maximal unknot}. \fullref{fig:trivial} shows a maximal unknot; the image of this under any contact isotopy
will also be a maximal unknot.

\begin{figure}[ht]
\centerline{\includegraphics[height=.7in]{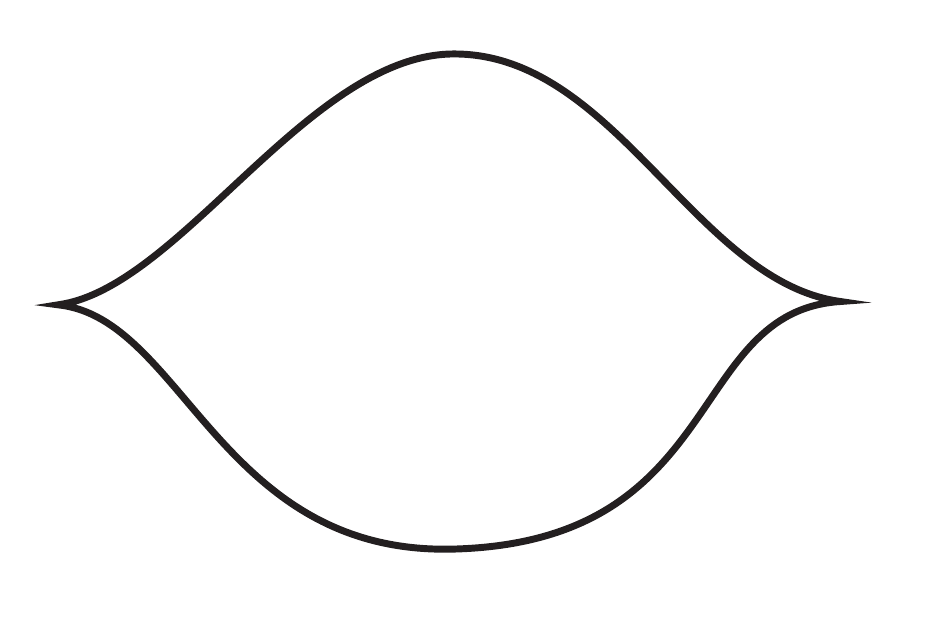}}
\caption{A maximal Legendrian unknot}
\label{fig:trivial}
\end{figure}

 We will now look at two families of Legendrian links that
can be constructed from maximal unknots. The first type, called a \emph{rational link}, 
is topologically a closure of a
rational tangle, and we use a vector notation to describe it.  Similar
notations for rational tangles and for another type of Legendrian link are
found in Ernst \cite{Ernst} and Traynor \cite{Generating}, respectively. We define the
link $L=\left( 2w_{n},k_{n},\dots ,2w_{1},k_{1},2w_{0}\right) ,w_{i},k_{i}>0$, recursively as follows. For $n=0$, the link has $2w_{0}$ ``horizontal''
crossings, as illustrated in \fullref{fig:integral}  for the case $w_{0}=2$.
For $n\geq 1$, the link $\left( 2w_{n},k_{n},\dots ,2w_{1},k_{1},2w_{0}\right) $ is formed from the
link $\left( 2w_{n},k_{n},\dots ,2w_{1}\right) $ by adding $k_{1}$
``vertical'' crossings and $2w_{0}$ ``horizontal'' crossings as shown in
\fullref{fig:recursive}. For example, see
\fullref{fig:(42212)} 
to see how we build the link $\left( 4,2,2,1,2\right) $
from the link $\left( 4,2,2\right) $.

\begin{figure}[ht] 
\labellist\small
\pinlabel {$\Lambda_1$} [bl] at 369 183
\pinlabel {$\Lambda_2$} [tl] at 362 74
\endlabellist
\centerline{\includegraphics[height=1in]{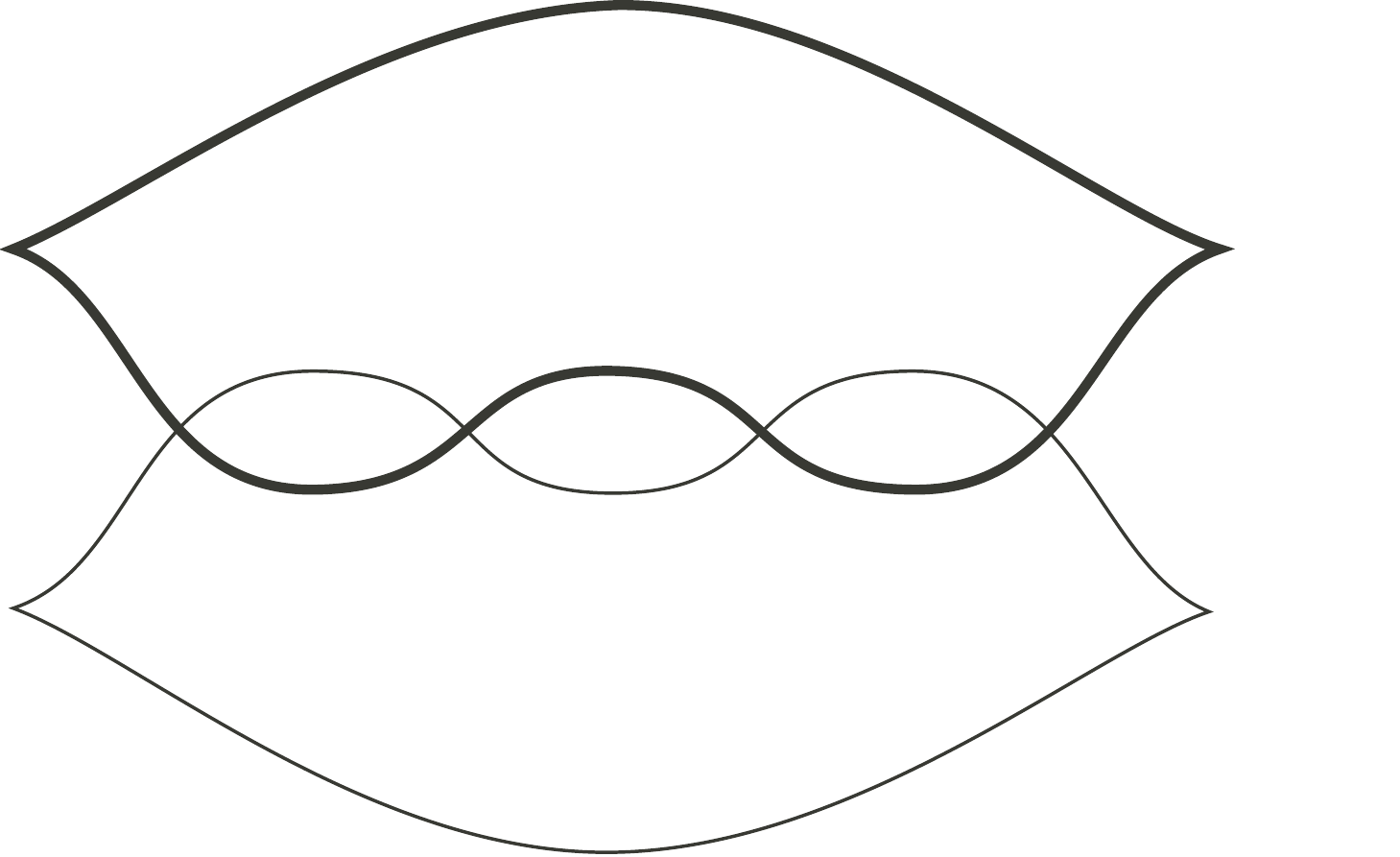}}
\caption{The Legendrian link $L=\left( 4\right)$}
\label{fig:integral}
\end{figure}

\begin{figure}[ht]
\labellist\tiny
\pinlabel {$(2w_n,k_n,\ldots,2w_1)$} at 110 162
\pinlabel {$2w_0$} at 325 172
\pinlabel {$k_1$} at 487 90
\pinlabel {$(2w_n,k_n,\ldots,2w_1)$} at 488 236
\endlabellist
\centerline{\includegraphics[height=2.3in]{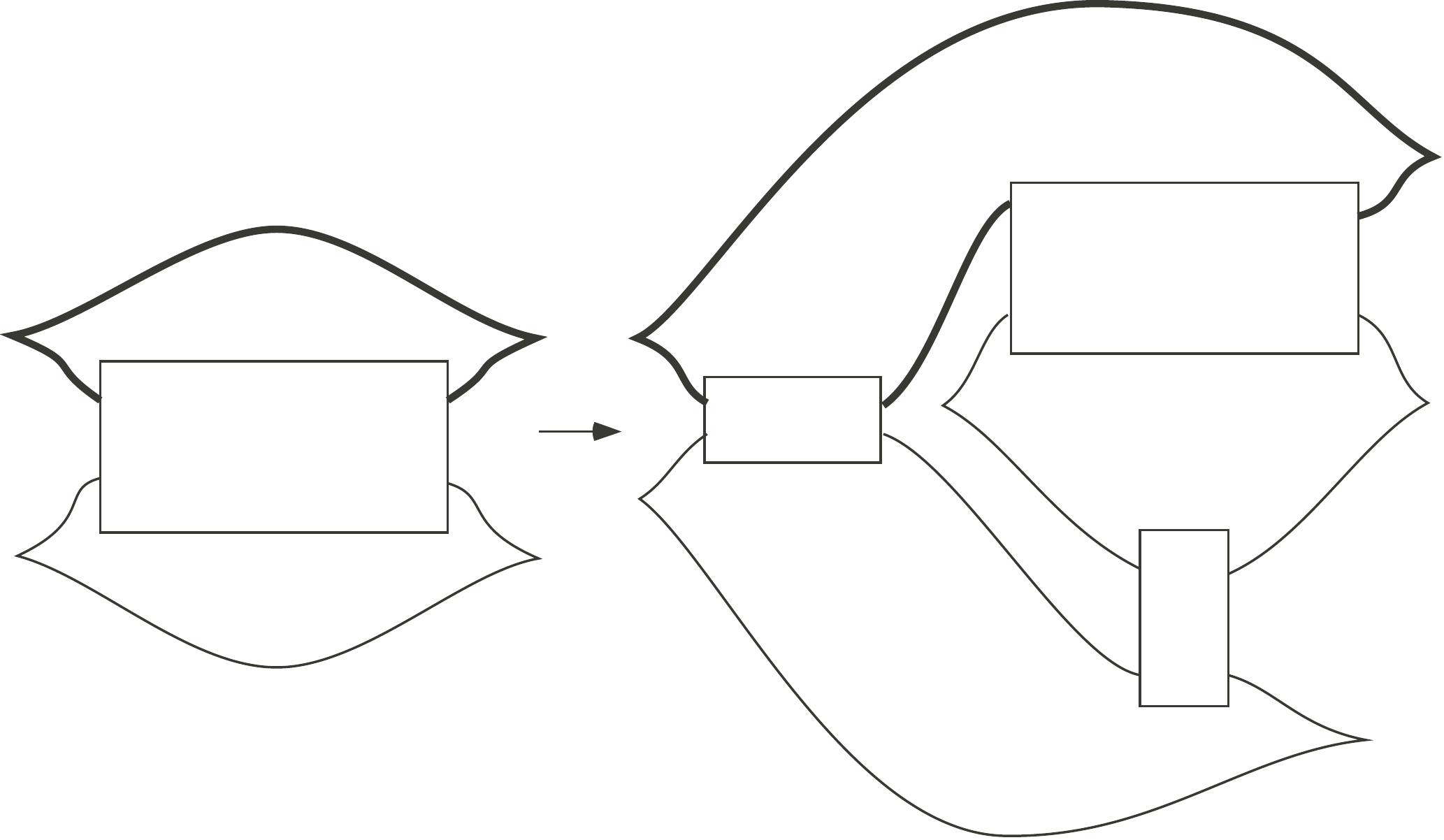}}
\caption{The link $\left( 2w_{n},k_{n},\dots ,2w_{1},k_{1},2w_{0}\right) $
is built from the link $\left( 2w_{n},k_{n},\dots ,2w_{1}\right)$.}
        \label{fig:recursive}
\end{figure}

\begin{figure}[ht]
\labellist\small
\pinlabel {$\Lambda_0$} [tr] at 33 74
\pinlabel {$\Lambda_1$} [br] at 34 182
\pinlabel {$\Lambda_0$} [tr] at 306 58
\pinlabel {$\Lambda_1$} [br] at 311 170
\pinlabel {$(4,2,2)$} [b] at 108 0
\pinlabel {$(4,2,2,1,2)$} [b] at 420 0
\endlabellist
\centerline{\includegraphics[height=1.5in]{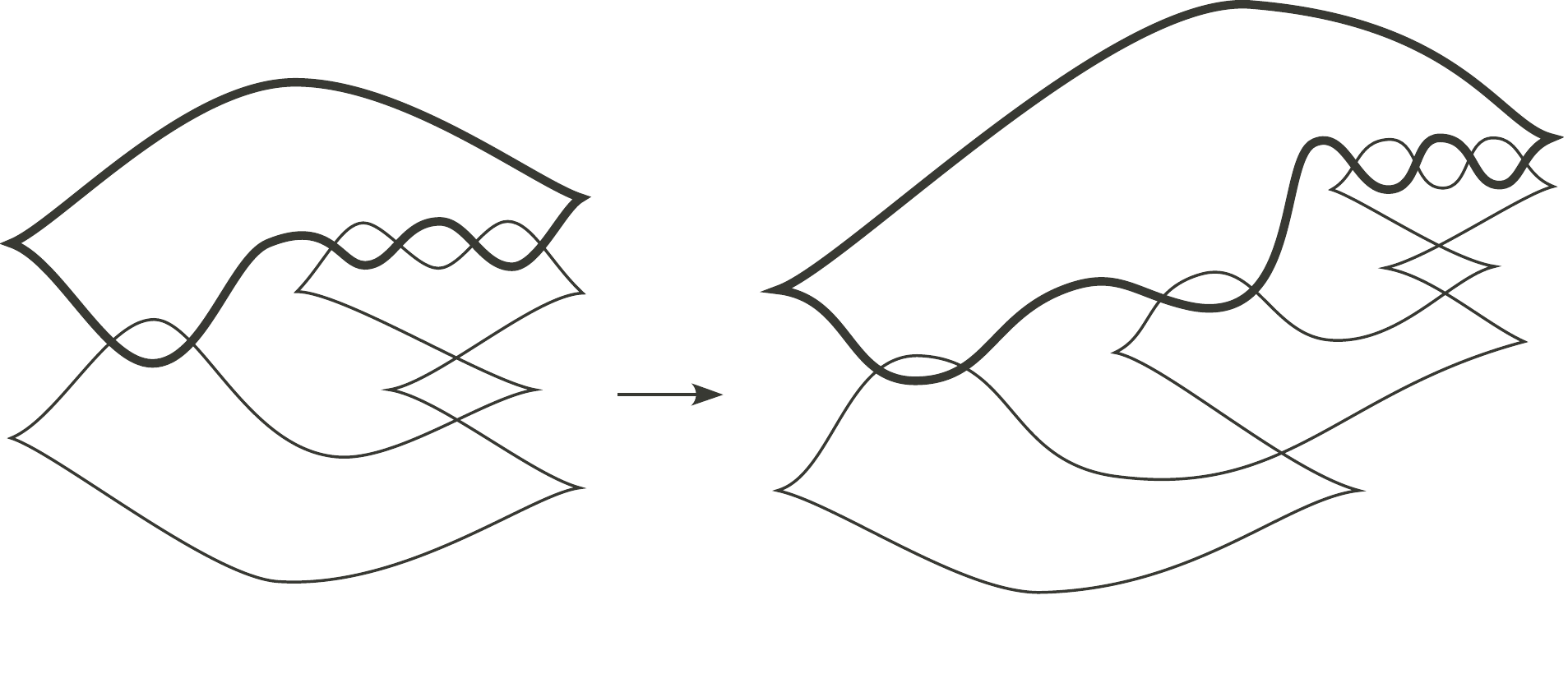}}
\caption{The link $\left( 4,2,2,1,2\right) $ is built from the link $\left(4,2,2\right)$.}
        \label{fig:(42212)}
\end{figure}

We will call a second type of link under consideration a \emph{twist link}. 
Topologically, twist links are
formed by clasping together two unknots, twisting each component a
number of times, and then clasping the knots together again. 
We will
let $L_{j,k}=\left( \Lambda _{1},\Lambda _{0}\right) $ represent the twist
link with left component ($\Lambda _{1}$) twisted so it has $j$ crossings,
and right component ($\Lambda _{0}$) twisted to have $k$ crossings; see
\fullref{fig:Ljk}. \label{Ljkdef}

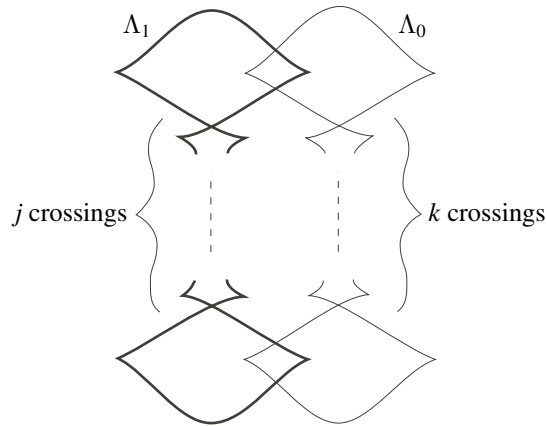
\begin{figure}[ht]
\labellist\small
\pinlabel {$\Lambda_1$} [br] at 102 434
\pinlabel {$\Lambda_0$} [bl] at 368 434
\pinlabel {$j$ crossings} [r] at 77 239
\pinlabel {$k$ crossings} [l] at 400 239
\endlabellist
\centerline{\includegraphics[height=2.2in]{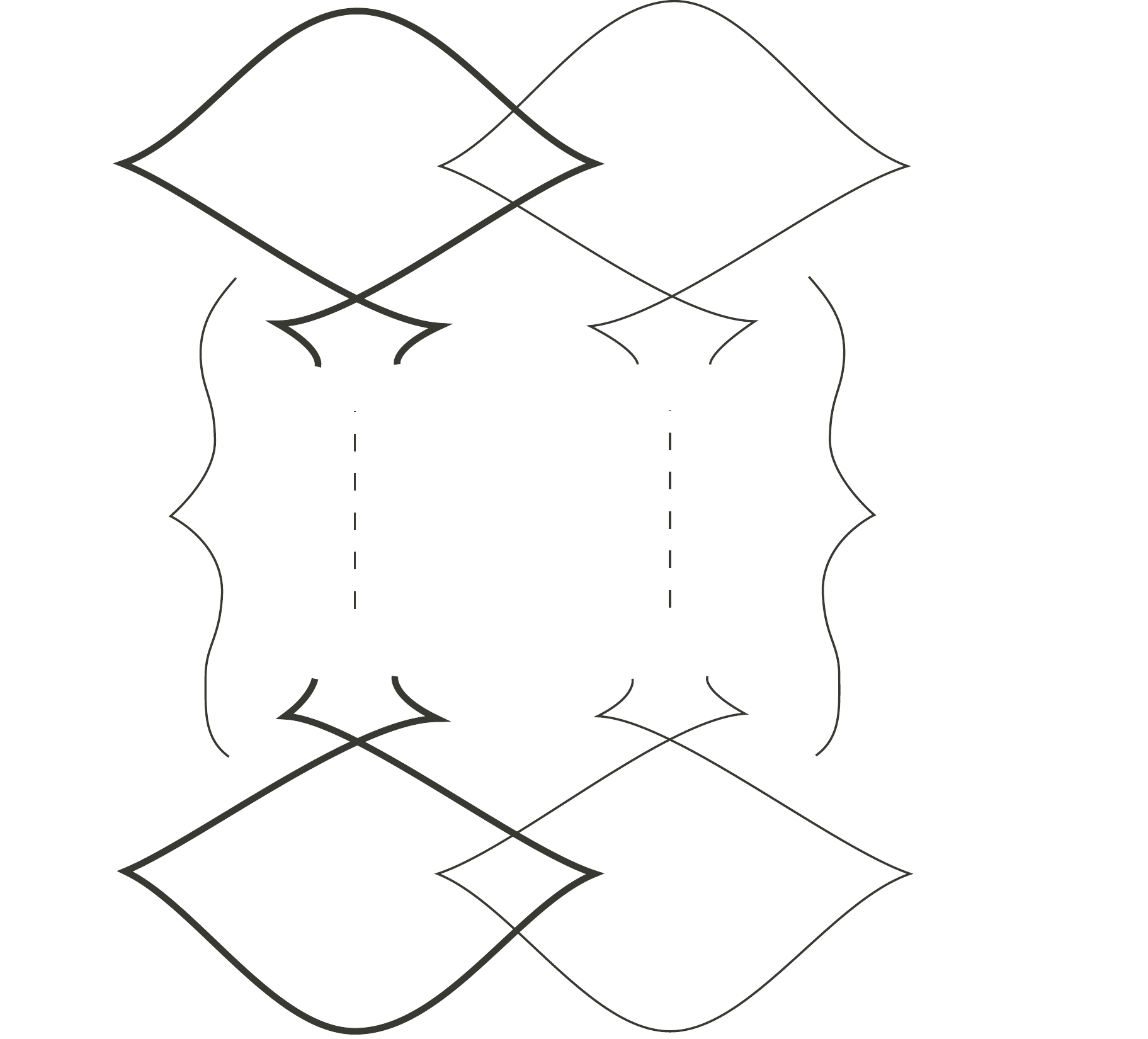}}
\caption{The twist link $L_{j,k}$}
\label{fig:Ljk}
\end{figure}

These links of maximal unknots will be studied via the technique of generating
families.   The type of generating families we use to study these links  
is different from the generating families used in Traynor \cite{Generating} to study
links of topologically nontrivial components in $\mathcal{J}^{1}\left( S^{1}\right) $. 
Previously Traynor
associated a \emph{quadratic at infinity generating family} to each link
component $\Lambda _{i}$. That is, given $\Lambda _{i}$, there exists a
function $F_{i} \co S^{1}\times \R^{N_{i}}\rightarrow \R$ for
some $N_{i}\geq 0$, such that 
\begin{equation*}
\Lambda _{i}=\left\{ \left( x,\partial _{x}F_{i}\left( x,v\right)
,F_{i}\left( x,v\right) \right) \mid \partial _{v}F_{i}\left( x,v\right)
=0\right\} \text{ and}
\end{equation*}
\begin{equation*}
F_{i}\left( x,v\right) \equiv Q\left( v\right) \text{ outside a compact set,}
\end{equation*}
where $Q$ is a nondegenerate quadratic function. However, it
is not hard to see that when a component $\Lambda _{i}$ is a knot in $\R^{3}$, it cannot be defined by a
function that is quadratic at infinity. We will see that it is sometimes
possible to define a Legendrian knot by a  ``linear-quadratic" at
infinity generating family, abbreviated as an LQ generating family. A
function $F_{i}\co \R\times \left( \R\times \R^{N_{i}}\right) \rightarrow \R$ is an \emph{LQ generating family for 
}$\Lambda _{i} \subset \jetR = \R^3$ if 
\begin{equation*}
\Lambda _{i}=\left\{ \left( x,\partial _{x}F_{i}\left( x,l,v\right)
,F_{i}\left( x,l,v\right) \right) \mid \partial _{l}F_{i}\left( x,l,v\right)
=0\text{ and }  \partial _{v}F_{i}\left( x,l,v\right) =0\right\} 
\end{equation*}
\begin{equation*}
\text{ and } F_{i}\left( x,l,v\right) \equiv J\left( l\right) +Q\left( v\right) \text{ outside a compact set,}
\end{equation*}
where $J$ is a nonzero linear function and $Q$ is a nondegenerate quadratic
function. 

In order to use LQ generating families to construct polynomial invariants,
we need existence and uniqueness results. The following three theorems, proved
in \fullref{Existence&Uniqueness}, are modeled after David Th\'{e}ret's
results in the symplectic category \cite{Theret}.

\begin{theorem}[Existence Theorem] 
Let $\Lambda$ be a maximal Legendrian unknot. Then $\Lambda$ has an LQ generating family.
\end{theorem}

\noindent An LQ generating family can be explicitly constructed for some
configurations of a maximal unknot, such as the one shown in \fullref{fig:trivial}. 
Since any maximal unknot can be isotoped to look like the one in
\fullref{fig:trivial}, 
the existence theorem is a direct result of the
following theorem, in the case where $\Delta _{n}$ is zero dimensional.

\begin{theorem}[Serre Fibration Structure Theorem]
\label{Serre}(See \fullref{Theret 4.2}) Let $\mathcal{F}$ be the set of
LQ generating families and let $\mathcal{L}$ be the set of Legendrian
submanifolds of $\mathcal{J}^{1}\left( M\right) $.
 Then the map $\pi \co\mathcal{F}\rightarrow \mathcal{L}$ is a smooth Serre fibration, up to
equivalence. More precisely, if the smooth map $f\co\Delta _{n}\rightarrow 
\mathcal{L}$ has a smooth lift $F\co \Delta _{n}\rightarrow \mathcal{F}$ and if 
$\left( f_{t} \co \Delta _{n}\rightarrow \mathcal{L}\right) _{t\in \left[ 0,1\right]}$ is a smooth homotopy of $f=f_{0}$, then there is a smooth
homotopy $\left( F_{t}\co \Delta _{n}\rightarrow \mathcal{F}\right) _{t\in \left[ 0,1\right]}$ such that $F_{0}=F$ up to equivalence (that is, up to
stabilization and fiber-preserving diffeomorphism), and $\pi \circ
F_{t}=f_{t}$ for every $t\in \left[ 0,1\right] $.
\end{theorem}

\noindent It is clear that, if a LQ generating family exists, it is not unique. 
However it is unique up to a certain equivalence.

\begin{theorem}[Uniqueness Theorem] (See \fullref{uniqueness}.)
Let $\Lambda $ be a maximal Legendrian unknot. Then all
LQ generating families for $\Lambda $ are equivalent up to stabilization and
fiber-preserving diffeomorphism.
\end{theorem}

\noindent The operations of stabilization and fiber-preserving diffeomorphism
are explained in \fullref{defn:gf_equivalence}.

In \fullref{Applications}, in analogy with \cite{Generating},
  we apply the existence and uniqueness results
from \fullref{Existence&Uniqueness} to  associate  polynomials,
$\Gamma^\pm_\Delta(\lambda)$, to a two-component
Legendrian link where each component is a maximal unknot. Normalized
versions of these polynomials give Laurent polynomials $\Gamma^\pm(\lambda)$ 
that are shown to be invariants of a Legendrian
link.  In fact, $\Gamma^+(\lambda)$ is determined by $\Gamma^-(\lambda)$:
$\Gamma^+(\lambda) = \lambda \cdot \Gamma^-(\lambda)$. The following theorems give the calculations of $\Gamma^-$  for rational links
and twist links.

\begin{theorem}[See \fullref{theorem 6.1}]
Let $L$ be the Legendrian link
$$\left(
2w_{n},k_{n},\dots ,2w_{1},k_{1},2w_{0}\right)$$
as described  in \fullref{fig:recursive}.
  Then 
  $$ \Gamma ^{-}\left( \lambda \right) \left[ L\right] = w_{0}\lambda
^{0}+w_{1}\lambda ^{-k_{1}}+w_{2}\lambda ^{-\left( k_{1}+k_{2}\right)
}+\dots +w_{n}\lambda ^{-\left( k_{1}+k_{2}+\dots +k_{n}\right)}. $$
\end{theorem}

\begin{theorem}[See \fullref{Ljkpolys}]
\label{Ljkpolys1}  Let $L_{j,k}$ be a
Legendrian twist link as described in \fullref{fig:Ljk}.  Then 
$$\Gamma ^{-}\left( \lambda \right) \left[ L_{j,k}\right] =\lambda
	^{0}+\lambda ^{-\left| j-k\right|}.$$
\end{theorem}

\noindent The topological link type of a twist link is dependent only on
the value of $j+k$.  So as a direct result of \fullref{Ljkpolys1}, we are
able to use generating family polynomials to distinguish several Legendrian twist links
with the same topological link type; see \fullref{Ljktype}.

Generating family polynomials can also be used to show that a link
is ordered.  $L = (\Lambda_1, \Lambda_0)$ is ordered if it is not
equivalent to $\overline L = (\Lambda_0, \Lambda_1)$.     In fact, a necessary condition
for a link to be unordered is that the $\Gamma^-$ polynomial
must be the same, up to a shift,  if $\lambda$ is replaced by $\lambda^{-1}$.
 
\begin{theorem}[See \fullref{order}]
Let $L=\left( \Lambda _{1},\Lambda _{0}\right) $ be a Legendrian link where $\Lambda _{1}$ and $\Lambda _{0}$
are maximal unknots.  If there does not exist an  $l\in \Z$ such that $\Gamma
^{-}\left( \lambda \right) \left[ L\right] =$ $\lambda ^{l}\cdot \Gamma
^{-}\left( \lambda ^{-1}\right) \left[ L\right] $, then the link $L$ is
ordered.
\end{theorem}

\noindent The rational links are always topologically unordered, but by the above
polynomial calculations we show:

\begin{theorem}[See \fullref{rat_order}]
Let $L=\left( \Lambda _{1},\Lambda _{0}\right) $ be the rational Legendrian
link 
$\left( 2w_{n},k_{n},\dots ,2w_{1},k_{1},2w_{0}\right) $.  If the
vector $\left( 2w_{n},k_{n},\dots ,2w_{1},k_{1},2w_{0}\right) $ is not
palindromic, then the link $L$ is ordered.
\end{theorem}

Another use of generating family polynomials is in distinguishing rational
links that differ by a Legendrian ``flyping" operation.  Flypes are discussed
in \fullref{poly_calcs}.
 Both horizontal and vertical flyping procedures, when
applied to a  Legendrian link, result in a link of the same topological link
type.  It is shown by Traynor \cite{Generating} that vertical flypes when applied to a 
Legendrian link $L = \left( 2w_{n},k_{n},\dots ,2w_{1},k_{1},2w_{0}\right) $
preserve the Legendrian link type.  However polynomial calculations for horizontal
flypes show that nonequivalent Legendrian links may be produced.  

\begin{theorem}[See \fullref{flypesthm}]
Let $L=\left( 2w_{n},k_{n},2w_{n-1}^{p_{n-1}},\dots
,k_{1},2w_{0}^{p_{0}}\right) $ be the Legendrian link obtained 
by doing $p_i$ horizontal flypes to the $w_i$ horizontal entry of the
rational link $(2w_n, k_n, \dots, k_1, 2w_0)$.  For $j=0,1,\dots ,n-1$, let $\sigma \left(
j\right) =1+\sum_{i=0}^{j}p_{i}$ mod $2$.  Then 
\begin{equation*}
\Gamma ^{-}\left( \lambda \right) \left[ L\right] =\lambda ^{m}\cdot \Bigl[
w_{0}\lambda ^{0}+\sum_{i=1}^{n}w_{i}\lambda ^{[ \left( -1\right)
^{\sigma \left( 0\right)}k_{1}+\dots +\left( -1\right) ^{\sigma \left(
i-1\right)}k_{i}]}\Bigr],
\end{equation*}
where $m$ is chosen so that $\Gamma ^{-}$ has degree zero.
\end{theorem}

\noindent For an example of a rational link that is not equivalent to a flyped version
of the link, see \fullref{fig:intro_flypes}.

\begin{figure}[ht]
\labellist
\pinlabel {$\not=$} at 283 85
\endlabellist
\centerline{\includegraphics[height=1in]{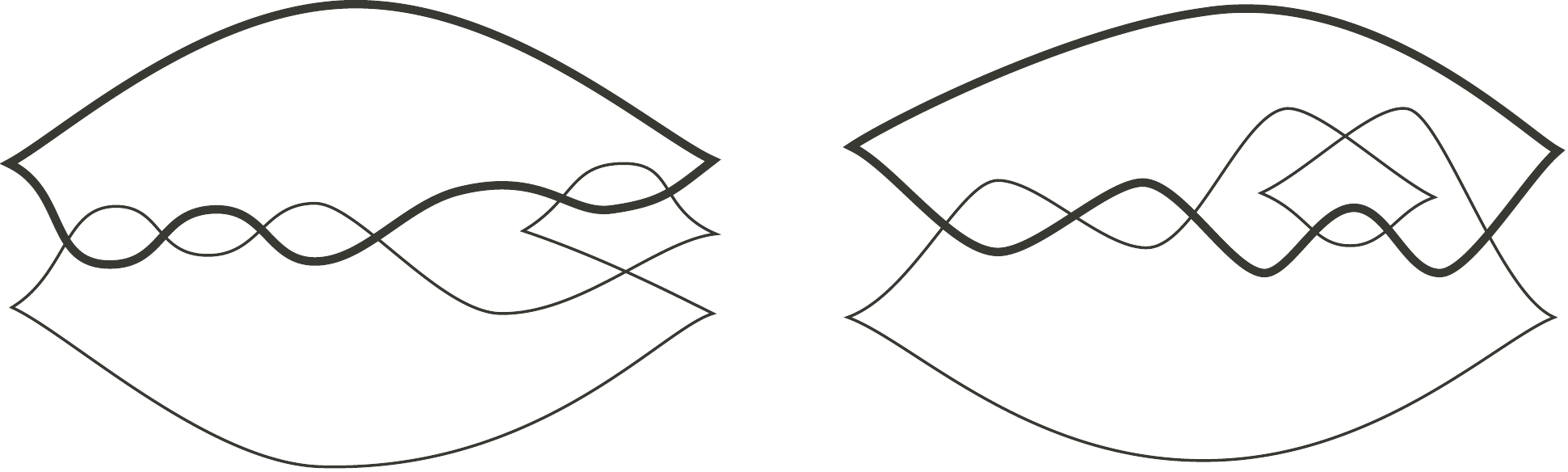}}
\caption{The link $\left( 2,1,4\right) $ is not equivalent to its flyped
version, $\left( 2,1,4^{1}\right)$}
        \label{fig:intro_flypes}
\end{figure}

Finally, we compare
generating family polynomials to decomposition number polynomials  and homology polynomials obtained from the DGA.  The decomposition number polynomials are quite different
from the generating function polynomials.  However,
for both rational links and twist links, the (negative) generating
family polynomials are the same as the (negative) DGA homology polynomials;
see \fullref{same_polys}.

\subsubsection*{Acknowledgements}  We are thankful to the members of the Philadelphia
Area Contact Topology (PACT) Seminar  for their useful comments during a mini-series presentation
of these results.  We are also grateful for the many careful and detailed suggestions of the  
reviewer.

\section{Background information}
\label{Background}

If $M$ is an $n$--manifold, then the $1$--jet space of $M$, $\mathcal{J}^{1}\left( M\right) 
= T^*(M) \times \R$,
 is a $\left( 2n+1\right)$--manifold with a 
contact structure $\xi _{\text{std}}$ on $\mathcal{J}^{1}\left( M\right)
=\left\{ \left( x,y,z\right) \right\} $ given by the kernel of the $1$--form $\alpha =dz-ydx$.   There are
no integral $j$--dimensional submanifolds of the $(2n{+}1)$--dimensional
$\jetM$ when $j > n$. However, there are numerous integral $n$--dimensional
submanifolds.  Such submanifolds are called \emph{Legendrian}.   For example,
for any $f\co M\rightarrow 
\R$, $j^{1}\left( f\right) :=
\bigl\{ \bigl( x, \frac{\partial f}{\partial x},f\left( x\right) \bigr)
\mid x\in M\bigr\} $
is a Legendrian submanifold of $\jetM$. 
We will pay special attention to the $3$--dimensional contact manifold
$\jetR = \R^3$.

  A \emph{Legendrian knot} is a closed and connected $1$--dimensional Legendrian submanifold  
 in a $3$--dimensional contact
manifold.   A \emph{Legendrian link} is the union
of one or more non-intersecting Legendrian knots.  
The \emph{front} of a Legendrian curve is its image under 
the  \emph{front projection\label{front projection}}
given by $\pi _{xz}\left( x,y,z\right) =\left( x,z\right) $.   
Given a Legendrian curve $L$ in $\jetR$,
 let $\pi _{xz}\left( L\right) =C$.    $C$ is an immersed curve with nonvertical tangents and
semi-cubic cusps (see \fullref{fig:cusp}), which generically has only
double points.  If $L$ is
free of self-intersections, any crossing of $C$ must be a transverse
intersection. 
 Conversely, any such curve in $\R^{2}$ 
determines a Legendrian curve in $\R^{3}$.
 We may use the fact that $y=\frac{dz}{dx}
$ for any $\left( x,y,z\right) \in L$, a condition imposed by the contact
structure, to recover the third coordinate of $L$.  

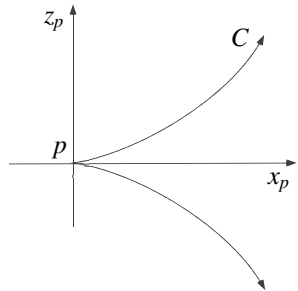
\begin{figure}[ht]\label{fig:cusp}
\labellist\small
\pinlabel {$z_p$} [r] at 72 305
\pinlabel {$p$} [br] at 72 143
\pinlabel {$x_p$} [t] at 305 143
\pinlabel {$C$} [br] at 277 270
\endlabellist
\centerline{\includegraphics[height=1.5in]{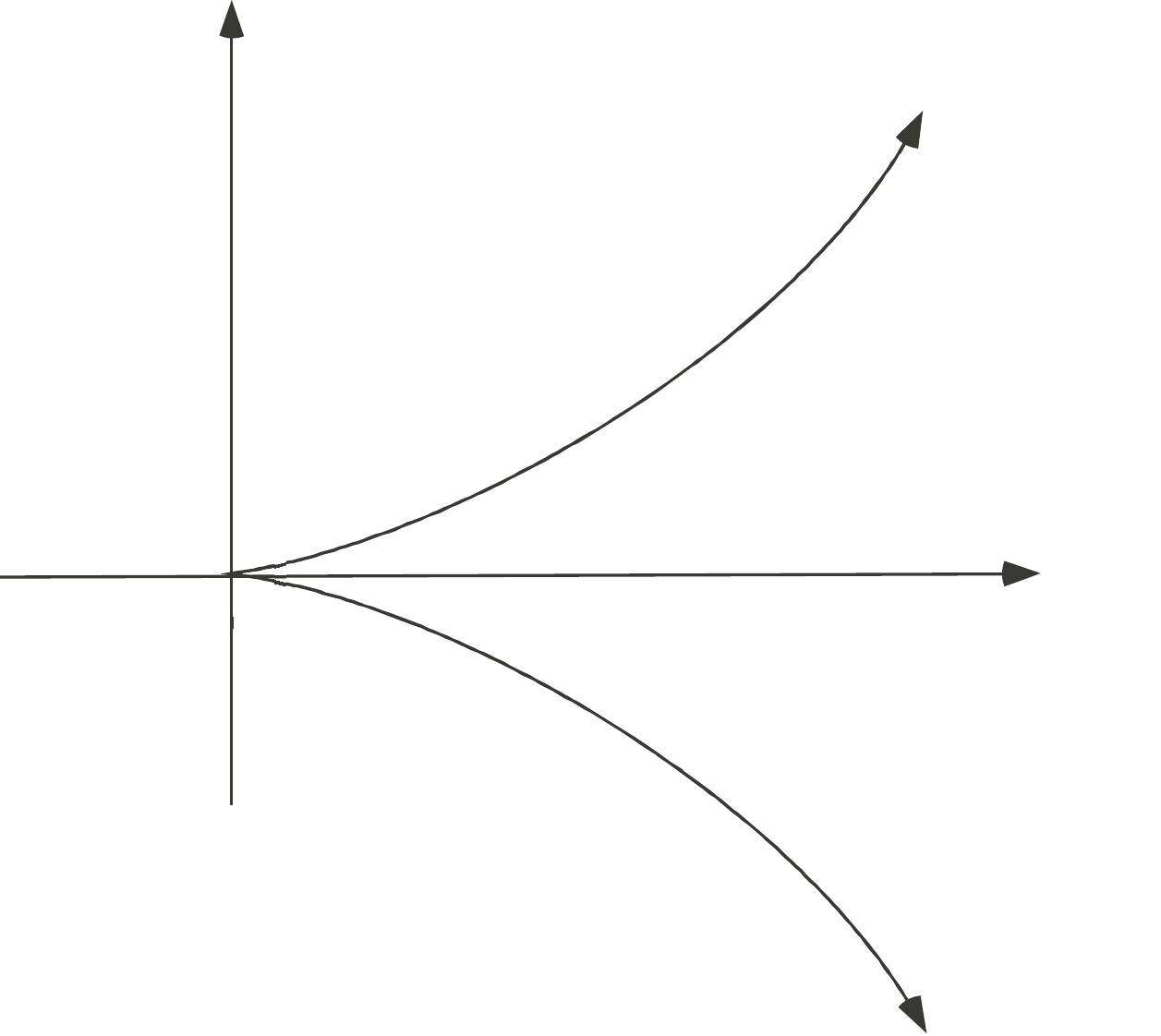}}
\caption{A semi-cubic cusp is a point $p\in C$ for which
there exist coordinates $x_{p},z_{p}$ such that $p=(0,0)$ and
${C=\bigl\{(x_{p},z_{p}) \mid z_{p}^{2}=\pm x_{p}^{3}\bigr\}}$
in a neighborhood of $p$.
}
\end{figure}

The main question in the study of Legendrian links is whether or not two
links are equivalent.  We say two links $L = (\Lambda_n, \dots, \Lambda_0)$, 
$\widetilde{L} = (\widetilde \Lambda_n, \dots, \widetilde \Lambda_0)$ are \emph{equivalent} if there exists a smooth, one-parameter family $L_{t}$ of
Legendrian links such that $L_{0}=L$ and $L_{1}=\widetilde{L}$, or, equivalently,
there exists a contact isotopy of the ambient space taking $\Lambda_i$ to $\widetilde \Lambda_i$
for $i = 0, \dots, n$.  Two topologically equivalent Legendrian links will be
distinct if, for some $i$, components $\Lambda_i$ and $\widetilde \Lambda_i$ do not have the same classical Legendrian integer invariants
given by the rotation and Thurston--Bennequin numbers.   For more background on
Legendrian knots and links, see, for example, Etnyre \cite{Etnyre}.  In the following, we will extend the work in \cite{Generating} and construct 
new invariants for Legendrian links using
the technique of generating families.

Generating families provide a way to encode a Legendrian submanifold of $\jetM$
by a real-valued function on $M\times \R^{k}$.
Suppose we have a smooth function $F\co M\times \R^{k}\rightarrow 
\R$, $\left( x,v\right) \mapsto F\left( x,v\right) ,$ such that $0$
is a regular value of the map 
$(\frac{\partial F}{\partial v_1}, \dots, \frac{\partial F}{\partial v_n})\co M\times \R^{k}\rightarrow \R^{k}$.  We define $\Sigma _{F}$, the \emph{critical locus} of $F$, as
\begin{equation*}
\Sigma _{F}:=\bigl\{ \left( x,v\right) \in M\times \R^{k}\mid \tfrac{\partial F}{\partial v_{i}}\left( x,v\right) =0\text{ for }i=1,2,\dots
,k\bigr\}.
\end{equation*}
By the preimage theorem, $\Sigma _{F}$ is a one-dimensional submanifold of 
$M\times \R^{k}$.  Define an immersion $i_{F} \co \Sigma
_{F}\rightarrow \mathcal{J}^{1}\left( M\right) $ by 
\begin{equation*}
i_{F}\left( x_{0},v_{0}\right) =\allowbreak \left( x_{0},\partial
_{x}F\left( x_{0},v_{0}\right) ,F\left( x_{0},v_{0}\right) \right) .
\end{equation*}
When $i_{F}$ is an embedding, $L:=i_{F}\left( \Sigma _{F}\right) $ is a
Legendrian submanifold of $\mathcal{J}^{1}\left( M\right) $.  We say that $F$ \emph{generates} $L$, or $F$ is the \emph{generating family} for $L$. 
In the following, we will start with a Legendrian submanifold $L\subset \mathcal{J}^{1}\left( M\right) $ and seek a generating family $F\co M\times \R^{k}\rightarrow \R$ for $L$.

Clearly the choice of a generating family for a given Legendrian $L\subset 
\mathcal{J}^{1}\left( M\right) $, if one exists, is not unique.  If $F\co M\times \R^{k}\rightarrow \R$ generates $L$ then so does,
for example, $F^{\prime}\co M\times \R^{k+1}\rightarrow \R$,
where $F^{\prime}\left( x,v_{1},\dots ,v_{k},v_{k+1}\right) =F\left(
x,v_{1},\dots ,v_{k}\right) +v_{k+1}^{2}$.  Therefore we wish to work with
equivalence classes of families rather than with the families themselves. 

\begin{definition} \label{defn:gf_equivalence}
Two generating families $F_i \co M\times \R^{k_i}\rightarrow \R$, $i = 1,2$
are \emph{equivalent}  if and only
if they can be made equal after a succession of  fiber-preserving diffeomorphisms
and stabilizations; these operations on the generating family  are defined as follows:

\begin{enumerate}
\item
Given a generating family $F\co M\times \R^{k}\rightarrow \R$, 
suppose $\Phi \co M\times \R^{k}\rightarrow M\times \R^{k}$ is
a fiber-preserving diffeomorphism, i.e.,  $\Phi \left( x,v\right) =\left( x,\phi_x \left(
v\right) \right) $ for diffeomorphisms $\phi_x $. Then $F^\prime =F \circ
\Phi$ is said to be obtained from $F$ by a \emph{fiber-preserving
diffeomorphism}.

\item
Given a generating family $F \co M\times \R^{k}\rightarrow \R$, 
let $Q \co \R^{j}\rightarrow \R$ be a quadratic function. 
Define $F^\prime \co M\times \R^{k}\times \R^{j}\rightarrow 
\R$ by $F^\prime \left( x,v_{1},v_{2}\right) =\left( F\oplus
Q\right) \left( x,v_{1},v_{2}\right) =F\left( x,v_{1}\right) +Q\left(
v_{2}\right) $. Then $F^\prime$ is said to obtained from $F$ by a
\emph{stabilization}.
\end{enumerate}
\end{definition}

If two families are equivalent, we can get from one
to another by performing one stabilization followed by one diffeomorphism
(see Th\'eret \cite{Theret}).

There is a parallel theory of generating families in the symplectic
category.  In a symplectic manifold, Lagrangian submanifolds are objects of
central importance, and the theory of generating families gives one a way to
encode some Lagrangians in $T^{\ast}\left( M\right) $ by a function $F\co M\times \R^{N}\rightarrow \R$.  In this version, the same
procedure is used to construct $\sigma _{F}$, but now the associated
immersion $i_{F}$ does not include the value of $F$. 

In both the contact and symplectic settings, these generating families are
defined on noncompact domains. Analytically it is convenient to consider
functions that are ``well-behaved'' outside of a compact set. A common
convention has been to consider generating families that are ``quadratic at
infinity.''  This means that outside of a compact set, $F\left( x,v\right)
=Q\left( v\right) $, where $Q$ is a nondegenerate quadratic function.
See, for example, Viterbo \cite{Viterbo} and Th\'eret \cite{Theret}.

Quadratic at infinity generating families can generate only particular
  Legendrian or Lagrangian submanifolds.  It is not hard
to see, for example, that the maximal Legendrian unknot pictured in \fullref{fig:trivial} does not have a quadratic at infinity generating family.  However,
this Legendrian knot will have a ``linear-quadratic" at infinity generating family,
abbreviated as an LQ generating family.  

\begin{definition}
A generating family $F\co M\times \R\times \R^{k}\rightarrow 
\R$, $k\geq 0$ is \emph{linear-quadratic at infinity} if we have $F\left( x,l,v\right) =J_{x}\left( l\right) +Q_{x,l}\left( v\right) $ outside
a compact set, where $J_{x}$ is a nonzero linear function of $l$ for each $x$ and $Q_{x,l}$ is a nondegenerate quadratic function of $v$ for each pair 
$\left( x,l\right) \in M\times \R$.  $F$ is \emph{special
linear-quadratic at infinity} if $F\left( x,l,v\right) =J\left( l\right)
+Q\left( v\right) $ outside a compact set for some nonzero linear function $J$ and some nondegenerate quadratic function $Q$.
\end{definition}

\noindent Note that the added requirement for a special LQ generating family
is that the linear and quadratic parts be independent of the base point $x\in M$.    In fact, any LQ generating family is equivalent to a special LQ generating
family.   This can be proved following the argument of the proof of
\cite[Proposition~2.12]{Theret}. 

\fullref{fig:gffig} sketches the graphs of two LQ generating families for the
maximal Legendrian unknot.  Note that the generating families differ by a stabilization, and
hence they are equivalent.

\begin{figure}[ht]
\labellist\small
\pinlabel {(a)} [b] at 120 0
\pinlabel {(b)} [b] at 410 0
\endlabellist
\centerline{\includegraphics[width=5in]{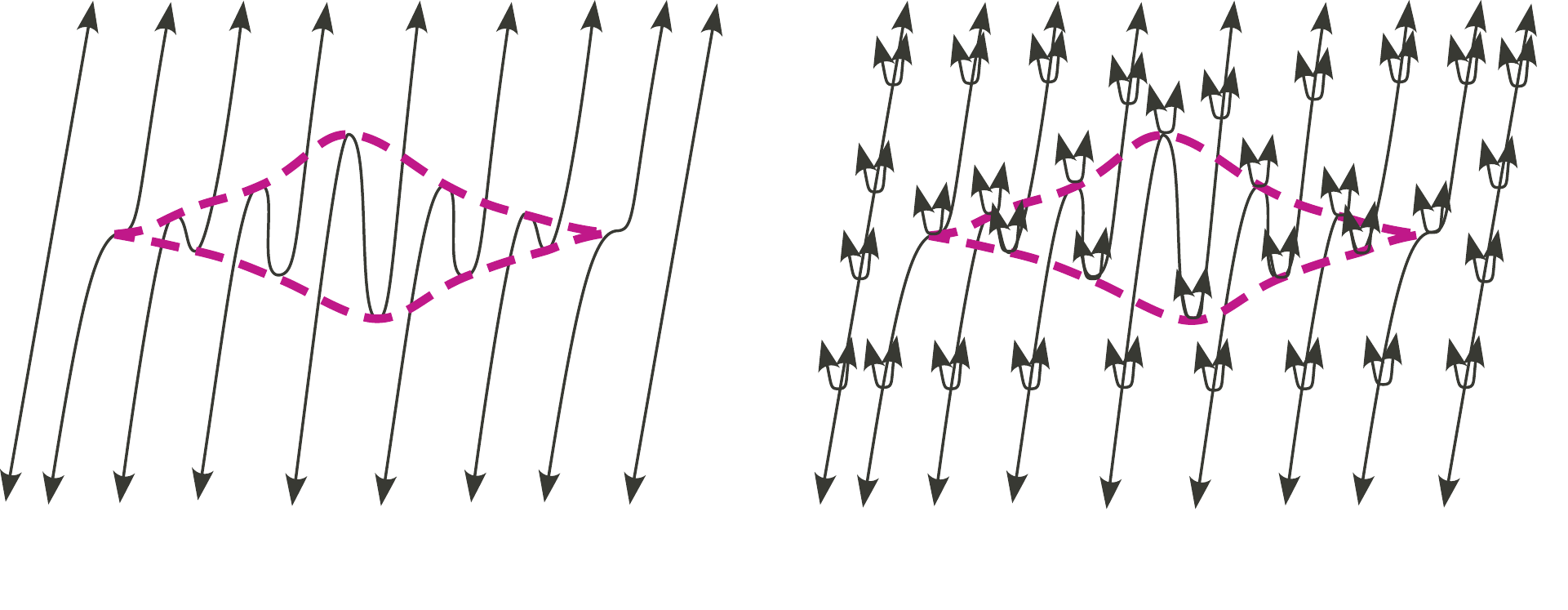}}
\caption{The LQ generating families shown in (a) and (b) generate the
same maximal unknot, which is sketched in with a broken line.
The generating family in (b) is a stabilization of the one in (a).}
\label{fig:gffig}
\end{figure}

\section{Existence and uniqueness of LQ generating families} \label{Existence&Uniqueness}
 
  In Sections  \ref{Persistence} and \ref{Uniqueness}, we prove that if we have a Legendrian $L\subset  \jetRm $ with a unique  LQ
generating family  then 
$\varphi _{1}\left(L\right) $ also has a unique LQ generating family for any
contact isotopy $\varphi _{t}$ of $\jetRm $. 
Sub\fullref{Unique_Unknot} contains the proof of uniqueness of an LQ
generating family for the maximal unknot in $\jetR = \R^3$ 
   pictured in \fullref{fig:trivial}.
 Th\'{e}ret's work
in \cite{Theret} forms a basis for the theorems and proofs in this section.
 
\subsection{Persistence of LQ generating families under isotopies}\label{Persistence}

Let us first introduce some notation for use in this subsection and the next. 
Let $M$ be $\R^m$ or $\Delta_n \times \R^{m-n} \subset \R^m$ where $\Delta_n$
is the standard $n$--simplex in $\R^n$.   Results in this section will also apply to
the case where $M$ is a  closed manifold but this setting will not be needed in this paper.
 Let $\mathcal{L}$
be the set of Legendrians in $\mathcal{J}^{1}\left( M\right) $.  For each integer $k\geq 0$, let $\mathcal{F}_{k}$ be
the set of all LQ generating families defined on $M\times \R\times 
\R^{k}$.  If $k$ is not specified we will simply use $\mathcal{F}$.
 Let $\pi \co \mathcal{F}\rightarrow \mathcal{L}$ be the map taking a
generating family $f$ to the Legendrian in $\mathcal{L}$ generated by $f$.

The main result of this subsection is that the map $\pi \co \mathcal{F}\rightarrow 
\mathcal{L}$ is a smooth Serre fibration (up to equivalence).  In
particular, this implies that if $\gamma \co \left[ 0,1\right] \rightarrow 
\mathcal{L}$ is a path and $\gamma \left( 0\right) $ has a generating family
(i.e. there exists $F\in \pi ^{-1}\left( \gamma \left( 0\right) \right) $),
then there is a lift of the path to $\tilde{\gamma} \co \left[ 0,1\right]
\rightarrow \mathcal{F}$ such that $\tilde{\gamma}(0) = F$ up to equivalence and
$\pi \left( \tilde{\gamma}\left( t\right)
\right) =\gamma \left( t\right) $ for all $t\in \left[ 0,1\right] $.  More
generally, if $f\co \Delta _{n}\rightarrow \mathcal{L}$ has a smooth lift $F\co \Delta _{n}\rightarrow \mathcal{F}$, then for any smooth homotopy $f_{t}\co \Delta _{n}\rightarrow \mathcal{L}$ of $f$ there exists a smooth
homotopy $F_{t}\co \Delta _{n}\rightarrow \mathcal{F}$ of (perhaps an
equivalent version of) $F$ 
satisfying $\pi \circ F_{t}=f_{t}$. We will prove this Serre fibration property in two
stages.  We will first show that the above path lifting property holds (see
\fullref{path lifting property}).  This follows from a Legendrian
version of ``Chekanov's formula".  We will then prove a lemma (\fullref{Theret 4.3}) that allows us to view a homotopy of Legendrians as a single
Legendrian in a larger space.  \fullref{path
lifting property}  will then be applied again.

\begin{theorem}
\label{path lifting property}Let $M$ be $\R^m$ or 
$\Delta_n \times \R^{m-n} \subset \R^m$, and let $\Lambda $
be a Legendrian submanifold of $\mathcal{J}^{1}\left( M\right)$.  Let $\left( \kappa _{t}\right) _{t\in \left[ 0,1\right]}$ be a compactly
supported contact isotopy of $\mathcal{J}^{1}\left( M\right)$, $\kappa_0 = \operatorname{id}$.  Assume
that $\Lambda $ has an LQ generating family $F\co M\times \R\times \R^{k}\rightarrow \R$.  Then
there exists an integer $j\geq 0$ and a path $\left( F_{t}\right) _{t\in \left[ 0,1\right]}$ of LQ generating families defined on $M\times \R\times \R^{k}\times \R^{j}$ such that

\begin{enumerate}
\item $F_{0}\left( x,l,v,w\right) =F\left( x,l,v\right) +Q\left( w\right) $,
where $Q$ is a nondegenerate quadratic function on $\R^{j};$

\item $F_{t}=F_{0}$ outside a compact set;

\item $F_{t}$ generates $\kappa _{t}\left( \Lambda \right) $ for $t\in \left[
0,1\right] $.
\end{enumerate}
\end{theorem}

\begin{proof}
\cite[Theorem~A.1]{CircularHelix} is an analogous theorem with
quadratic (rather than linear-quadratic) at infinity conditions.  A careful
check of the equation shows that the proof of Theorem A.1 generalizes to the
LQ situation.  The following is a brief summary of the proof.

The first step is to realize that even when $M = \Delta_n \times \R^{m-n}$
 we can work in the setting of $\mathcal{J}^{1}\left( \R^{m}\right) $ rather than in  $\mathcal{J}^{1}\left(
M\right) $. $\Delta_n \times \R^{m-n}$ naturally lies in $\R^m$.  By an extension of
 $F$ to 
$\R^{m}\times \R\times \R^{k}$, we have that $\Lambda 
$ embeds into a Legendrian submanifold $\Lambda ^{m}\subset
\mathcal{J}^{1}\left( \R^{m}\right) $.  As shown in
\cite[Proposition~A.2]{CircularHelix}, $\kappa _{t}$ extends to a compactly supported contact
isotopy $\kappa _{t}^{m}$ of $\jetRm$, and to show $\kappa
_{t}\left( \Lambda \right) $ has an LQ generating family it suffices to
prove that $\kappa _{t}^{m}\left( \Lambda ^{m}\right) \subset \mathcal{J}^{1}\left( \R^{m}\right) $ has an LQ generating family.

Next we translate the problem into a symplectic situation by looking at a
certain $\R^{+}$--equivariant Lagrangian $\mathcal{L}_{\Lambda}$ in $T^{\ast}\left( \R^{m}\times \R^{+}\right) $ that
corresponds to a Legendrian $\Lambda \subset \mathcal{J}^{1}\left(
\R^{m}\right) $ (see \cite[Equation~A.3]{CircularHelix}).  If $\Lambda
\subset \mathcal{J}^{1}\left( \R^{m}\right) $ has an LQ generating
family, then so does $\mathcal{L}_{\Lambda}\subset T^{\ast}\left( \R^{m}\times \R^{+}\right) $.  

Next we associate a symplectic diffeomorphism $\Psi _{\kappa}$ of $T^{\ast
}\left( \R^{m}\times \R^{+}\right) $ to a contact
diffeomorphism $\kappa$ isotopic to the identity.  This has a corresponding
Lagrangian submanifold $\widetilde{\Gamma}_{\Psi _{\kappa}}$ of $T^{\ast
}\left( \R^{2m+1}\times \R^{+}\right) $.  When the
diffeomorphism $\Psi _{\kappa}$ is sufficiently close to the identity,
these Lagrangians have simple generating families.  Then by ``Chekanov's
formula,'' generating families for $\mathcal{L}_{\Lambda}$ and $\widetilde{\Gamma}_{\Psi _{\kappa}}$ can be ``composed'' to obtain a generating
family for $\Psi _{\kappa}\left( \mathcal{L}_{\Lambda}\right)
=\mathcal{L}_{\kappa \left( \Lambda \right)}$ (see
\cite[Proposition~A.5]{CircularHelix} or \cite[Section~4]{Symplectic}).  A careful analysis of these
equations shows that if the generating family for $\mathcal{L}_{\Lambda}$
is LQ, then the family for $\Psi _{\kappa}\left( \mathcal{L}_{\Lambda
}\right) $ will also be LQ.  By breaking the isotopy $\kappa _{t}$ into a
composition of small diffeomorphisms, this procedure produces an LQ
generating family for $\mathcal{L}_{\kappa _{t}\left( \Lambda \right)}$ for
all $t$.  See, for example, \cite[Section 4]{Symplectic}.

Lastly, \cite[Proposition~A.6]{CircularHelix} shows that if $G$ is an
LQ generating family for $\mathcal{L}_{\Lambda}\subset T^{\ast}\left( 
\R^{m}\times \R^{+}\right) $, then a ``slice'' of $G$ will
be an LQ generating family for $\Lambda $. Hence since $\mathcal{L}_{\kappa
_{t}\left( \Lambda \right)}$ has an LQ generating family for all $t\in \left[ 0,1\right] $, $\kappa _{t}\left( \Lambda \right) $ has an LQ
generating family, say $F_{t}$, for all $t\in \left[ 0,1\right] $.  We can
see through checking the proofs in \cite{CircularHelix} that the remaining
conditions on $F_{t}$ are satisfied.\hfill
\end{proof}

The next lemma shows, in particular, that from a family of Legendrian
submanifolds of $\jetR$ parameterized by points in $\Delta _{n}$ one can
construct a single, $\left( n+1\right)$--dimensional Legendrian submanifold of $\mathcal{J}^{1}\left( \Delta _{n}\times \R \right) $.  The lemma is the Legendrian
version of \cite[Lemma~4.3]{Theret}. Note that when we say $S\co \Delta _{n}\times M\times \R\times \R^{k}\rightarrow \R$ is linear-quadratic at
infinity, we mean it is linear in the $\R$ variable and quadratic in
the $\R^{k}$ variable.

\begin{lemma}
\label{Theret 4.3}Note first that $\mathcal{J}^{1}\left( \Delta _{n}\times
M\right) \cong T^{\ast}\left( \Delta _{n}\right) \times \mathcal{J}^{1}\left( M\right) $. \newline
Let $\rho \co \mathcal{J}^{1}\left( \Delta _{n}\times
M\right) \rightarrow \mathcal{J}^{1}\left( M\right) $ be the associated
projection.
\begin{enumerate}
\item[(a)]  Let $f\co \Delta _{n}\times B \rightarrow \mathcal{J}^{1}\left( M\right) $
be a differentiable map such that each $f_{a}$, defined by $f_{a}\left(
s\right) =f\left( a,s\right) $ is a Legendrian embedding of $B$ in $\mathcal{J}^{1}\left( M\right) $.  Then there is a map $v\co \Delta _{n}\times
B \rightarrow \left( \R^{n}\right) ^{\ast}$ such that 
 $$
\qquad L:=\left\{ \left( a,v\left( a,s\right) ,f_{a}\left( s\right) \right) \mid
a\in \Delta _{n},s\in B \right\} \subset T^{\ast}\left( \Delta _{n}\right)
\times \mathcal{J}^{1}\left( M\right)
$$
is a Legendrian submanifold in $\mathcal{J}^{1}\left( \Delta _{n}\times
M\right) $.  Furthermore, if $S_{a}\co M\times \R\times \R^{k}\rightarrow \R$ is a smooth family of LQ generating families of $L_{a}:=f\left( \left\{ a\right\} \times B \right) $, then the total family $S\co \Delta _{n}\times M\times \R\times \R^{k}\rightarrow 
\R$ is an LQ generating family for the Legendrian $L$ given above.

\item[(b)]  Conversely, if $F\co \Delta _{n}\times B \rightarrow \mathcal{J}^{1}\left(
\Delta _{n}\times M\right) $ is a Legendrian embedding, transversal to the
spaces $W_{a}:=\left\{ a\right\} \times \left( \R^{n}\right) ^{\ast
}\times \mathcal{J}^{1}\left( M\right) $ and having an LQ generating family $S\co \Delta _{n}\times M\times \R\times \R^{k}\rightarrow 
\R$, then $L_{a}:=\left( \rho \circ F\right) \left( \left\{
a\right\} \times M\right) $ is a Legendrian in $\mathcal{J}^{1}\left(
M\right) $ with LQ generating family $S_{a}=S\bigl( \left\{ a\right\} \times
M\times \R\times \R^{k}\bigr) $.
\end{enumerate}
\end{lemma}

\begin{proof}
First, $\mathcal{J}^{1}\left( \Delta _{n}\times M\right) \cong T^{\ast
}\left( \Delta _{n}\right) \times \mathcal{J}^{1}\left( M\right) $ by a
reordering of coordinates as follows:  $\left( a,x,w,y,z\right) \sim \left(
a,w\right) \times \left( x,y,z\right) $ for $\left( a,w\right) \in T^{\ast
}\left( \Delta _{n}\right) $ and $\left( x,y,z\right) \in \mathcal{J}^{1}\left( M\right) ,$ since the contact form on $T^{\ast}\left( \Delta
_{n}\right) \times \mathcal{J}^{1}\left( M\right) $ and $\mathcal{J}^{1}\left( \Delta _{n}\times M\right) $ is given by $\alpha =dz-ydx-wda$.

To verify (a),  let 
$f\co \Delta _{n}\times B \rightarrow \mathcal{J}^{1}\left( M\right) $, 
$\left( a,x\right) \mapsto f_{a}\left( x\right) $ be a differentiable map
so that each $f_{a}$ is a Legendrian embedding of $B$ in $\mathcal{J}^{1}\left( M\right) $.    Write $f_{a}\left( s\right) =\left( f_{a}^{x}\left(
s\right) ,f_{a}^{y}\left( s\right) ,f_{a}^{z}\left( s\right) \right) $. 
Then if we let $v\left( a,s\right) =\frac{\partial f_{a}^{z}}{\partial
a}-f_{a}^{y}\left( s\right) \cdot \frac{\partial f_{a}^{x}}{\partial a}$, 
 $$
L =\left\{ \left(   a,f_{a}^{x}\left( s\right) ,v\left( a,s\right) ,f_{a}^{y}\left(
s\right) ,f_{a}^{z}\left( s\right)  \right) \right\} \subset \mathcal{J}^{1}\left( \Delta _{n}\times
M\right) 
 $$
is a Legendrian submanifold.    
It is straightforward to check that if
 $S_{a}\co M\times \R\times \R^{k}\rightarrow 
\R$ is a smooth family of LQ generating functions of $L_{a}:=f_{a}\left( M\right) $, $a\in \Delta _{n}$ then
the total function $S\co \Delta _{n}\times M\times \R\times \R^{k}\rightarrow \R$ is an LQ generating family for $L$ as given
above.   

To verify (b), let $F\co \Delta _{n}\times B \rightarrow \mathcal{J}^{1}\left( \Delta
_{n}\times M\right) $ be a Legendrian embedding transversal to the spaces $W_{a}:=\left\{ a\right\} \times \left( \R^{n}\right) ^{\ast}\times 
\mathcal{J}^{1}\left( M\right) $ and having an LQ generating function $S\co \Delta _{n}\times M\times \R\times \R^{k}\rightarrow 
\R$.  For $a\in \Delta _{n}$, let $L_{a}:=\left( \rho \circ
F\right) \left( \left\{ a\right\} \times M\right) $ and let
$S_{a}:=S\bigl(
\left\{ a\right\} \times M\times \R\times \R^{k}\bigr) $. 
Let $L=F\left( \Delta _{n}\times M\right) $.  Since $L$ is transversal to
each $W_{a}$, $L\cap W_{a}$ is a  
submanifold of $\mathcal{J}^{1}\left( \Delta _{n}\times M\right) $ that projects to $L_{a} \subset \jetM$.
It is straight forward to show that $S_{a}$ generates $L_{a}$.
\end{proof}

We are now ready to state and prove the main result of this subsection.  The
theorem and its proof are nearly identical to Th\'{e}ret's Theorem 4.2 and its proof in \cite{Theret}.    

\begin{theorem}
\label{Theret 4.2}The map $\pi \co \mathcal{F}\rightarrow \mathcal{L}$ is a
smooth Serre fibration, up to equivalence.  More precisely, if the smooth
map $f \co \Delta _{n}\rightarrow \mathcal{L}$ has a smooth lift $F\co \Delta
_{n}\rightarrow \mathcal{F}$ and if $\left( f_{t}\co \Delta _{n}\rightarrow 
\mathcal{L}\right) _{t\in \left[ 0,1\right]}$ is a smooth homotopy with $f=f_{0}$, then there is a smooth homotopy $\left( F_{t}\co \Delta
_{n}\rightarrow \mathcal{F}\right) _{t\in \left[ 0,1\right]}$ such that $F_{0}=F$ up to equivalence, and $\pi \circ F_{t}=f_{t}$ for every $t\in \left[ 0,1\right] $.
\end{theorem}

\begin{proof}
 Applying (a) from \fullref{Theret 4.3}, for each $t\in \left[ 0,1\right] $, the family $f_{t}\left( \Delta
_{n}\right) $ parameterized by points in $\Delta _{n}$ can be used to
construct a Legendrian $L_{t}$ in $\mathcal{J}^{1}\left( \Delta _{n}\times
M\right) $.  Thus we get a path $\bar{f}\co \left[ 0,1\right] \rightarrow 
\mathcal{L}\left( \mathcal{J}^{1}\left( \Delta _{n}\times M\right) \right) $
whose initial point $\bar{f}_{0}$ admits a generating family $\bar{F}_{0}$.
By part (b) of \fullref{Theret 4.3}, it is sufficient to prove that
we can lift the path $\bar{f}$ from the initial point $\bar{F}_{0}$.  By
\fullref{path lifting property}, we can lift the path $\bar{f}$, as
desired. 
\end{proof}

\subsection{Uniqueness of LQ generating families} \label{Uniqueness}

We say a Legendrian $L$ has the \emph{uniqueness property} if
any two of its LQ generating families are equivalent.  Here we prove that
if $L$ has the uniqueness property and $L_{1}$ is obtained from $L$ through
a Legendrian isotopy, then $L_{1}$ has the uniqueness property as well.  We
first prove a lemma which allows us to get a path in $\pi ^{-1}\left(
L_{1}\right) $ between any two LQ generating families of $L_{1}$. This
lemma corresponds to \cite[Lemma~5.2]{Theret}.

\begin{lemma} 
\label{Theret 5.2}Suppose that $L_0 \in \mathcal{L}$ is a Legendrian with the
uniqueness property and that $L_{1}$ is obtained from $L$ through a
Legendrian isotopy.  Let $f$ and $f^{\prime}$ be two LQ generating
families for $L_{1}$.  Then up to equivalence, $f$ and $f^{\prime}$ can be
connected by a path in $\pi ^{-1}\left( L_{1}\right) $.
\end{lemma}

Since the proof is  identical to that in \cite{Theret}, only the following sketch 
is given.  Let $f_0$ and $f_0^\prime$ be two LQ generating
families for $L_1$.  From the path of Legendrians between $L_0$ and
$L_1$, we construct a contractible loop of Legendrians based at $L_1$.
Using \fullref{path lifting property} (the path
lifting property) and the fact that $L_0$ has the uniqueness property, we know
this loop is covered by a path of generating families with endpoints
at (equivalent versions of) $f_0$ and $f_0^\prime$.  Since the loop
is contractible, by \fullref{Theret 4.2} (with $n=1$)
we get our desired result.

We now get the following uniqueness theorem, which corresponds to
\cite[Theorem~5.1]{Theret}.  

\begin{theorem}
\label{Theret 5.1}Let $L_{0}$ be a Legendrian with the uniqueness property.
 Suppose that $L_{1}=\varphi _{1}\left( L_{0}\right) $ where $\left(
\varphi _{t}\right) _{t\in \left[ 0,1\right]}$ is a Legendrian isotopy of $\mathcal{J}^{1}\left( M\right) $.  Then any two LQ generating families of $L_{1}$ are equivalent.
\end{theorem}

Again, since the proof is nearly identical to the proof of
\cite[Theorem~5.1]{Theret},
we will only sketch the argument.  If $f$ and $f^\prime$ are two LQ generating
families for $L_1$, then we know by \fullref{Theret 5.2}, up to equivalence, 
$f$ and $f^{\prime}$ can be
connected by a path in $\pi ^{-1}\left( L_{1}\right) $.  Therefore it
suffices to show that if $\left( f_{t}\right) _{t\in \left[ 0,1\right]}$ is
a smooth path of LQ generating families that generate a fixed Legendrian,
then there exists a fiber-preserving isotopy $\Phi_t$ such that
$f_t \circ \Phi_t = f_0$ for all $t \in [0,1]$.  By differentiating this equation
with respect to $t$, we get an equation for the vector field $X_t$ that
generates this isotopy.  It is easy to find a solution for this $X_t$ outside
the fiber critical set $\Sigma_t = \Sigma$ of $f_t$.  We then apply Hadamard's Lemma
to find a solution for $X_t$ near $\Sigma$.  These two solutions are
then pasted together by the choice of an appropriate bump function.

\subsection{Uniqueness of the LQ generating family for a basic unknot} \label{Unique_Unknot}

So far we have proved that existence and uniqueness of LQ generating
families persist under Legendrian isotopies.  We will now prove that the
maximal unknot in $\jetR$ shown in \fullref{fig:trivial} 
has the uniqueness property.  The proof of this theorem is philosophically
the same as \cite[Theorem~6.1]{Theret} where it is
proved that the Lagrangian $0$--section of $T^{\ast}M$, for $M$ an arbitrary
closed manifold, has a unique quadratic at infinity generating family.  Some
differences between the proofs are pointed out in \fullref{remark} after
the proof.

\begin{theorem}
\label{uniqueness}Let $L\subset \jetR = \R^3 $ be a
maximal Legendrian unknot.  Then any two LQ generating families of $L$ are
equivalent.
\end{theorem}

\begin{proof}
By \fullref{Theret 5.1}, it suffices to prove the theorem in the case
where $L$ is the unknot pictured in \fullref{fig:trivial}.  Let $f,$ $g$
be two LQ generating families for $L$.  By applying fiber-preserving
diffeomorphisms and stabilizations, we can assume the following:

\begin{itemize}
\item $f$ and $g$ have the same domain 
$\R \times \R\times \R^{n} = \{ (x, l, v) \}$; we will let $f_{x},g_{x} \co \R\times \R^{n}\rightarrow 
\R$ denote the associated fiber functions.

\item The critical points of $f_{x}$ agree with the critical points of $g_{x}$; we will let $C_{x}\subset \R\times \R^{n}$ denote
this set of critical points.  Assuming that the two cusp points of $\pi_{xz}(L)$
occur at $(0,0)$ and $(1,0)$, we see that $C_{0}$ and $C_{1}$ consist of a
single point with critical value $0$, $C_{x}$ for $x\in \left( 0,1\right) $
consists of two nondegenerate critical points with nonzero critical values,
and $C_{x}=\emptyset $ for all other $x$.  By a Morse theoretic argument
(see, for example, the proof of \fullref{proposition 3.12}), for $x\in \left( 0,1\right) $, we can write $C_{x}=\left\{ v_{0}\left( x\right)
,v_{1}\left( x\right) \right\} $ where $v_{1}\left( x\right) $ has index $k+1 $, $v_{0}\left( x\right) $ has index $k$, and
\begin{equation*}
f_{x}\left( v_{1}\left( x\right) \right) =g_{x}\left(
v_{1}\left( x\right) \right) >0> f_{x}\left(
v_{0}\left( x\right) \right) =g_{x}\left( v_{0}\left( x\right) \right) \text{.}
\end{equation*}

\item $f$ and $g$ are special; that is, outside a compact set of $\R\times \R^{n}$ they are strictly linear-quadratic, with linear and
quadratic parts independent of the base point $x\in \R$.
\end{itemize}

We next show that for all $x$, we can assume $f_{x}=g_{x}$ on a neighborhood 
$U\left( x\right) $ of $C_{x}$ in $\R\times \R^{n}$.

\begin{lemma} \label{nbhd}
There exists a neighborhood $U$ of $\cup _{x\in \R}C_{x}$ in $\R \times 
\R\times \R^{n}$ so that for $U\left( x\right) =U\cap \left(
\left\{ x\right\} \times \R\times \R^{n}\right) $, after 
fiber-preserving diffeomorphisms $f_{x}=g_{x}$ on $U\left( x\right) $, and all gradient trajectories of $f_{x}$
and  of $g_{x}$ intersect $U\left( x\right) $ in a connected set.
\end{lemma}

\begin{proof} From Arnol'd, Guse{\u\i}n-Zade and Varchenko \cite{Arnold}, we know that after applying
fiber-preserving
diffeomorphisms, in a neighborhood of $C_{0}$ in $\R \times \R\times \R^{n}$ we can assume $f\left( x,l,v\right) =g\left(
x,l,v\right) =l^{3}-xl+Q\left( v\right) $, where $Q$ is a nondegenerate
quadratic function.  A similar statement holds in a neighborhood of $C_{1}$.  Thus there is the desired neighborhood of $C_{0}\cup C_{1}$.

To see that there is a such a neighborhood for $\cup _{x\in \R}C_{x}$, we
first note that with a generic choice of metrics, $f_{x}$ (and $g_{x}$), $x\in \left( 0,1\right) $, forms a family of functions whose gradient flows
satisfy the Morse-Smale conditions.  We will now argue that for all $x\in
\left( 0,1\right) $, there is a single isolated gradient trajectory of $f_{x} $ from $v_{1}\left( x\right) $ to $v_{0}\left( x\right) $.  Since the
gradient trajectories of $f_{x}$ satisfy the Morse-Smale conditions, for all 
$x\in \left( 0,1\right) $, $W^{u}\left( v_{1}\left( x\right) \right) $ (the
unstable manifold of $v_{1}\left( x\right) $), $W^{s}\left( v_{0}\left(
x\right) \right) $ (the stable manifold of $v_{0}\left( x\right) $), and $f_{x}^{-1}\left( 0\right) $ intersect transversally in a finite number of
points.  Since near $x=0$, this intersection consists of a single point,
for all $x$ this intersection consists of a single point.  It follows that
for all $x\in \left( 0,1\right) $ there is a single gradient trajectory of $f_{x}$ in $W^{u}\left( v_{1}\left( x\right) \right) \cap W^{s}\left(
v_{0}\left( x\right) \right) $.  The analogous argument shows that there is
a single gradient trajectory of $g_{x}$ from $v_{1}\left( x\right) $ to $v_{0}\left( x\right) $ for all $x\in \left( 0,1\right) $.

It is not hard to show that by applying a diffeomorphism of $\R^{1+n}$, we can assume $f_{x}=g_{x}$ on
neighborhoods $V_{0}\left( x\right) $, $V_{1}\left( x\right) $ of $v_{0}\left( x\right) $ and $v_{1}\left( x\right) $, respectively.  This can be proved
using the Morse Lemma to obtain a diffeomorphism equating $f_x$ and $g_x$ 
on neighborhoods of $v_0$ and $v_1$ 
and then extending this diffeomorphism to all of $\R^{1+n}$ via the isotopy extension theorem
using the fact that any embedding of two disjoint balls must be isotopic to the identity.

By applying this neighborhood diffeomorphism together with the fact
that there is a single gradient trajectory from $v_{1}\left( x\right) $ to $v_{0}\left( x\right) $, after applying diffeomorphisms, we can assume that
\begin{itemize}
\item $v_{0}\left( x\right) $, $v_{1}\left( x\right) \in \R\times
\left\{ 0\right\} $ and $I\left( x\right) :=W^{u}\left( v_{1}\left( x\right)
\right) \cap W^{s}\left( v_{0}\left( x\right) \right) \subset $ $\R\times \left\{ 0\right\}$;

\item For neighborhoods $V_{1}\left( x\right) $, $V_{0}\left( x\right) $ of $v_{1}\left( x\right) $, $v_{0}\left( x\right) $, $f_{x}=g_{x}$ on $V_{1}\left( x\right) \cup I\left( x\right) \cup V_{0}\left( x\right)$.
\end{itemize}

Now, there exist tubular neighborhoods $T_{f}\left( x\right) $ and $T_{g}\left( x\right)$ of the open interval $I\left( x\right) $ consisting
of gradient trajectories of $f_{x},g_{x}$ that intersect $V_{1}\left(
x\right) $ and $V_{0}\left( x\right) $.  Viewing each tubular neighborhood
as a family of parameterized disks, we see that after applying a
diffeomorphism, we can assume $T_{f}\left( x\right) =T_{g}\left( x\right)
=T\left( x\right) $ and $f_{x}=g_{x}$ on $U\left( x\right) =V_{0}\left(
x\right) \cup T\left( x\right) \cup V_{1}\left( x\right) $.

The desired $U$ can be constructed from the above described neighborhoods of 
$C_{0}$, $C_{1}$, and $C_{x}$, $x\in \R$.  This completes the proof of
the lemma.\hfill \hfill 
\end{proof}

We will be using the gradient flows of $f_{x}$ and $g_{x}$ to define the
diffeomorphism $\varphi _{x}$ of $\R\times \R^{n}$ such that 
$f_{x}\circ \varphi _{x}=g_{x}$.  In particular, we will be working with
the orbits of the gradient flows.

To become familiar with the construction, suppose we have a situation where
every orbit of the gradient flow of both $f_{x}$ and $g_{x}$ intersect $U\left( x\right) $.  In this case, our diffeomorphism $\varphi _{x}$ is
defined by leaving all points in $U\left( x\right) $ fixed, while mapping
points outside of $U\left( x\right) $ in the following way.  For $w\notin
U\left( x\right) $, flow along the orbit of $w$ by the positive (negative)
gradient flow of $g_{x}$ until you reach a reference point $w^{\prime}\in
U\left( x\right) $.  Then flow from $w^{\prime}$ by the negative
(positive) gradient flow of $f_{x}$ until you reach a point $w^{\prime
\prime}$ such that $f_{x}\left( w^{\prime \prime}\right) =g_{x}\left(
w\right) $:  since $f_{x}=g_{x}$ on $U\left( x\right) $, it is not
difficult to verify that the $f_{x}$--orbit containing $w^{\prime}$ takes on
the same values as the $g_{x}$--orbit containing $w^{\prime}$ and thus $w^{\prime \prime}$ must exist.  Define $\varphi _{x}\left( w\right)
=w^{\prime \prime}$, so then $f_{x}\circ \varphi _{x}=g_{x}$.

To see that this map is well-defined, suppose we choose a different
reference point $\widetilde{w^{\prime}}$ in $U\left( x\right) $.  Then $w^{\prime}$ and $\widetilde{w^{\prime}}$ are in the same orbit with
respect to the gradient flow of $g_{x}$.  Moreover since $w^{\prime},
\widetilde{w^{\prime}}\in U\left( x\right) $ and $f_{x}=g_{x}$ on $U\left(
x\right) $, we also know that $w^{\prime},\widetilde{w^{\prime}}$ are in
the same orbit with respect to the gradient flow of $f_{x}$.  Thus they
both result in the same $w^{\prime \prime}$, so $\varphi _{x}$ is
well-defined.

In practice, we will usually have to consider the case where not every orbit
intersects $U\left( x\right) $.  In this case, we will see that every orbit
will intersect either $U\left( x\right) $ or a ``negative infinity level
set,'' \label{negative infinity}where a negative infinity level set for $f$
(respectively for $g$) is defined to be $f_{x}^{-1}\left( c\right) $
(respectively $g_{x}^{-1}\left( c\right) $) for some fixed $c\ll 0$. 
 \label{definitions}Let $L_{-\infty}^{f}:=f_{x}^{-1}\left( c\right) $ and let $L_{-\infty}^{g}:=g_{x}^{-1}\left( c\right) $. 
Since $f$ and $g$ are assumed to be special, $L_{-\infty}^{f}$ and $L_{-\infty}^{g}$ do
not depend on $x$.
In fact, since $f$ is
linear-quadratic at infinity, 
\begin{equation*}
\begin{array}{lll}
L_{-\infty}^{f} & = & \bigl\{ \left( l,v_{1}, \dots ,v_{n}\right) \in 
\R^{1 + n} \mid l-v_{1}^{2}-\cdots
-v_{j}^{2}+v_{j+1}^{2}+\cdots +v_{n}^{2}=c\bigr\} \\ 
& = & \bigl\{ \bigl( c+v_{1}^{2}+\cdots +v_{j}^{2}-v_{j+1}^{2}-\cdots
-v_{n}^{2},v_{1},v_{2},\dots ,v_{n}\bigr) \in \R^{1+n} \bigr\} .
\end{array}
\end{equation*}
Therefore $L_{-\infty}^{f}$ is an embedded image of $\R^{n}$ in $\R\times \R^{n}$, as is $L_{-\infty}^{g}$. To see that each orbit of the gradient flows of $f_{x}$
($g_{x}$) intersects $U\left( x\right) $ or $L_{-\infty}^{f}$ ($L_{-\infty
}^{g}$), observe that an orbit will either terminate at a critical point or
enter the region of $\R\times \R^{n}$ where $f_{x}$ ($g_{x}$) is standard linear-quadratic function.  In the first case the orbit
intersects $U\left( x\right) $, and in the second case the orbit intersects $L_{-\infty}^{f}$ ($L_{-\infty}^{g}$).

The idea now is to define a diffeomorphism similar to the one above, but
using reference points in $U\left( x\right) \cup L_{-\infty}^{f}$ and $U\left( x\right) \cup L_{-\infty}^{g}$.  However we must be careful to be
sure that the map is well-defined with respect to the reference point
chosen.  The difficulty here is when we have orbits of the gradient flows
of $f_{x}$ ($g_{x}$) that intersect both $U\left( x\right) $ and $L_{-\infty
}^{f}$ ($L_{-\infty}^{g}$), and thus we have reference points in both $U\left( x\right) $ and $L_{-\infty}^{f}$ ($L_{-\infty}^{g}$).  We will
now show that the map is, in fact, well defined.

Shrink $U(x)$ slightly to a closed set $D(x) \subset U(x)$ that satisfies
the conditions on $U(x)$ specified by \fullref{nbhd}.  
Consider the set of orbits that intersect both $D(x) \subset U\left( x\right) $ and $L_{-\infty}^{f}$ (or $L_{-\infty}^{g}$).  Let $W^{f}\left( x\right) $ ($W^{g}\left( x\right) $) denote a transverse slice of these orbits so that
each orbit with respect to the gradient flow of $f_{x}$ ($g_{x}$) intersects 
$W^f(x)$ ($W^g(x)$) precisely once.  In fact, $W^{f}\left( x\right) $ ($W^{g}\left(
x\right) $) can be chosen to be closed $n$--dimensional disks.
 Since $f_{x}=g_{x}$ on $D(x) \subset U\left( x\right) $, it is possible to choose $W\left( x\right) :=W^{f}\left( x\right) =W^{g}\left( x\right) \subset
D(x) \subset U\left( x\right) $.  For $w\in W\left( x\right) $, let $w_{g},w_{f}$ \label{definitions2}be the orbits of $w$ with respect to the gradient flows of $g_{x},f_{x}$ respectively.  Note $w_{g}\cap L_{-\infty}^{g}$ and $w_{f}\cap L_{-\infty}^{f}$ are each single points, call them $w_{g}^{-\infty}$ and $w_{f}^{-\infty}$ respectively.  
Let $W_1(x) =\cup _{w\in W(x)}w_{g}^{-\infty}$, $W_{2}(x)=\cup
_{w\in W(x)}w_{f}^{-\infty}$, and consider $\theta_x \co W_1(x) \to W_2(x)$
defined by $\theta_x \left(
w_{g}^{-\infty}\right) =w_{f}^{-\infty}$ for each $w\in W\left( x\right) $. 
Then by an application of the isotopy extension theorem, $\theta_x$
extends to a diffeomorphism $\Theta_x \co L_{-\infty}^{g}\rightarrow L_{-\infty}^{f}$.

Now we proceed to define $\varphi _{x}$.  As before, $\varphi _{x}$ leaves
all points in $D(x) \subset U\left( x\right) $ fixed.  For $w \notin D\left( x\right) $,
flow along the orbit of $w$ by the positive (negative) gradient flow of $g_{x}$ until you reach a reference point $w^{\prime}\in D\left( x\right)
\cup L_{-\infty}^{g}$.  If $w^{\prime}\in D\left( x\right) $, let $w^{\prime \prime}=w^{\prime}$.  If $w^{\prime}\in L_{-\infty}^{g}$, let 
$w^{\prime \prime}=\Theta _{x}\left( w^{\prime}\right) $. Now flow from $w^{\prime \prime}$ by the negative (positive) gradient flow of $f_{x}$
until you reach a point $w^{\prime \prime \prime}$ such that $f_{x}\left(
w^{\prime \prime \prime}\right) =g_{x}\left( w\right) $.  Define $\varphi
_{x}\left( w\right) =w^{\prime \prime \prime}$, so that $f_{x}\circ \varphi
_{x}=g_{x}$.  It is straight forward to verify that $\varphi_x$ is well-defined.

Now for each $x\in \R$ we have a diffeomorphism $\varphi _{x}$ of $\R\times \R^{n}$ such that $f_{x}\circ \varphi _{x}=g_{x}$.
 By construction, $\varphi _{x}$ varies smoothly with $x$.  Thus we have a
diffeomorphism $\Phi $ of $\R \times \R\times \R^{n}$
such that $f\circ \Phi =g$.  Hence $f$ and $g$ are equivalent.\hfill
\end{proof}

\begin{remark}
\label{remark}The above proof and Th\'{e}ret's proof of uniqueness of
generating families for the Lagrangian $0$--section have some differences. 
In Th\'{e}ret's setting, the construction of the set $U\left( x\right) $
containing the fiber-critical point where the function is standard is an
easy consequence of Morse Theory.  Also, in Th\'{e}ret's work, he uses a
$-\epsilon$--level set where above we use a $-\infty$--level set.  When
using the $-\epsilon$--level set, it is more immediate that the fiber
diffeomorphisms constructed via gradient flows are well defined.  However,
Th\'{e}ret must spend a great deal of effort to prove that for two
generating families $f$ and $g$, there is a global diffeomorphism between $f^{-1}\left( -\epsilon \right) $ and $g^{-1}\left( -\epsilon \right) $ that
is a diffeomorphism on each fiber slice.  This difficulty is avoided in the
above proof with the use of $-\infty$--level sets.\hfill \hfill
\end{remark}

\section{LQ generating family polynomials} \label{Applications}

We now use the existence and uniqueness results of the previous
section to define invariant polynomials for two component Legendrian
links in $\R^3$ where each component is a maximal unknot.
Much of this material parallels the results from Traynor \cite{Generating};
however, in that paper links live in $\jetS$ and each component of the link was Legendrian isotopic to 
the $1$--jet of a function. 

\begin{definition}
Given two functions $f_{i}\co \R \times \R^{1 + n_i} \rightarrow \R$, $i=0,1$, let 
$\Delta \co \R \times  \R^{1+ n_1} \times \R^{1+n_0} \rightarrow \R$ be given by $\Delta \left(
x,l_{1},v_{1},l_{0},v_{0}\right) =f_{1}\left( x,l_{1},v_{1}\right)
-f_{0}\left( x,l_{0},v_{0}\right) $.  Then $\Delta $ is the \emph{difference function} of $f_{1}$ and $f_{0}$.  Furthermore, if $f_{1}$ and $f_{0}$ generate $\Lambda _{1}$ and $\Lambda _{0}$, respectively, then $\Delta $ is a \emph{difference function of} $L=\left( \Lambda _{1},\Lambda
_{0}\right) $.
\end{definition}

Note the following facts about the critical points of a difference function $\Delta $.  First, if $\Delta $ is a difference function for a link $L=\left( \Lambda _{1},\Lambda _{0}\right) $, then the critical points of $\Delta $ are in one-to-one
correspondence with points of the form $\left( \left( x_{0},y_{0},z_{1}\right) ,\left(
x_{0},y_{0},z_{0}\right) \right) \in \Lambda _{1}\times \Lambda _{0}$. 
This can be seen by calculating the derivative of $\Delta $ in terms of the
derivatives of $f_{1}$ and $f_{0}$, where $f_{i}$ generates $\Lambda _{i}$ ($i=0,1$).  Therefore a critical point of $\Delta $ can be identified in the
front projection of $L$ as a pair of points, one point on each $\Lambda _{i}$, where the points have the same $x$ coordinate and the same slope. Second, 
if $L=\left( \Lambda _{1},\Lambda _{0}\right)$ is a link, the components of
$L$ do not intersect, and thus  $0$
is not a critical value of $\Delta $.

We will now proceed to define homology groups for $\Delta $, where $\Delta
\co \R \times \R^{1+n_1} \times   
\R^{1 + n_{0}}\rightarrow \R$ is a difference function for $L$.
 For $c\in \R$, $c$ a noncritical value of $\Delta $, let 
\begin{equation*}
\Delta ^{c}:=\left\{ \left( x,l_{1}, v_1, l_{0}, v_{0}\right) \mid \Delta
\left( x,l_{1}, v_1, l_{0}, v_{0}\right) \leq c\right\} \text{.}
\end{equation*}
Note that for every link $L$ and difference function $\Delta $, there exists
some constant $m>0$ such that all the critical values of $\Delta $ are in
the interval $\left[ -m+\epsilon ,m-\epsilon \right] $ for some $\epsilon >0$.  Then we define
\begin{equation*}
\Delta ^{\infty}:=\Delta ^{m}\text{,} \qquad \Delta ^{-\infty}:=\Delta ^{-m}\text{.}
\end{equation*}
Now the homology groups for $\Delta $ are defined as follows.

\begin{definition} 
Let $\Delta $ be a difference function for a Legendrian link $L=\left(
\Lambda _{1},\Lambda _{0}\right)$.  The \emph{total, positive, and negative homology groups of} $L$ \emph{
with respect to} $\Delta $ are defined as
\begin{eqnarray*}
H_{k}\bigl( \Delta \bigr) &=&H_{k}\bigl( \Delta ^{\infty},\Delta ^{-\infty
}\bigr) \text{,} \\
H_{k}^{+}\bigl( \Delta \bigr) &=&H_{k}\bigl( \Delta ^{\infty},\Delta
^{0}\bigr) \text{,} \\
H_{k}^{-}\bigl( \Delta \bigr) &=&H_{k}\bigl( \Delta ^{0},\Delta ^{-\infty
}\bigr) \text{,}
\end{eqnarray*}
for $k\in \Z$, where the relative homology groups are calculated
with coefficients in $\Z/2\Z$.
\end{definition}

If $\Delta $ and $\Delta ^{\prime}$ are two difference functions for a link 
$L$, then their homology groups will be related as follows.  If $\Delta 
$ and $\Delta ^{\prime}$ differ only by a fiber-preserving diffeomorphism,
then they will have the same homology groups.  However if $\Delta $ and $\Delta ^{\prime}$ differ by a stabilization, then there exists  
$i\in \Z$ such that 
$$
H_{k+i}\left( \Delta \right) =H_{k}\left( \Delta ^{\prime}\right) , \quad
H_{k+i}^{+}\left( \Delta \right) =H_{k}^{+}\left( \Delta ^{\prime}\right)
, \quad
H_{k+i}^{-}\left( \Delta \right) =H_{k}^{-}\left( \Delta ^{\prime}\right).
$$
In fact, the total homology groups do not carry any information about
a particular link.  

\begin{proposition}
\label{proposition 3.12}Let $L=\left( \Lambda _{1},\Lambda _{0}\right)
 $ be a Legendrian link  of maximal unknots.  Given a difference
function $\Delta $ for $L$, $H_{k}\left( \Delta \right) \simeq 0$ for all $k\in \Z$.
\end{proposition}

To prove this proposition, we will use the following lemma.  The lemma is similar to
\cite[Lemma~3.10]{Generating}, and can be proved following
the argument presented there.

\begin{lemma}
\label{lemma 3.10}  Consider a smooth 1--parameter family of difference
functions $\Delta _{t}\co \R \times \R^{1+n_1} \times \R^{1+n_{0}}\rightarrow \R$, $t\in \left[ 0,1\right] $, where each $\Delta _{t}$ is the difference of two LQ
generating families.  Given paths $\alpha ,\beta \co \left[ 0,1\right]
\rightarrow \R$ such that, for all $t$, $\alpha \left( t\right)
,\beta \left( t\right) $ are noncritical values of $\Delta _{t}$ with $\alpha \left( t\right) <\beta \left( t\right) $.  Then for any $t\in \left[
0,1\right] $ and $k\in \Z$, $H_{k}\bigl( \Delta _{0}^{\beta \left(
0\right)},\Delta _{0}^{\alpha \left( 0\right)}\bigr) \simeq H_{k}\bigl(
\Delta _{t}^{\beta \left( t\right)},\Delta _{t}^{\alpha \left( t\right)
}\bigr) $.
\end{lemma}

\noindent Using this, we now prove \fullref{proposition 3.12}:
 \begin{proof} 
Recall that if $\Delta $ and $\Delta ^{\prime}$ are two difference
functions for $L$, then  there exists some $i\in \Z$ such
that $H_{k+i}\left( \Delta \right) \simeq H_{k}\left( \Delta ^{\prime
}\right) $.  Therefore if the theorem is true for one particular difference
function of $L$, then it is true for every difference function.

By hypothesis, each strand of $L$ can be individually isotoped so that it is
the maximal unknot shown in \fullref{fig:trivial}. Choose 
isotopies of the strands $\Lambda_1^t, \Lambda_0^t$, $t \in [0,1]$, so that
$\Lambda_1^0 = \Lambda_1$, $\Lambda_0^0 = \Lambda_0$, and when
$t = 1$, $\Lambda_1^1$ and $\Lambda_0^1$ are maximal unknots translated
so that  $\left( x_{1},y_{1},z_{1}\right) \in \Lambda _{1}^{1}$ and 
$\left( x_{0},y_{0},z_{0}\right) \in \Lambda _{0}^{1}$  implies that $x_{1}\neq x_{0}$; in other words, $\Lambda _{1}^{1}$ and $\Lambda _{0}^{1}$
have no base points in $\R$ in common.  Consider $L_t = (\Lambda_1^t, \Lambda_0^t)$, 
$t \in [0, 1]$.  If $L$ is  
a nontrivial link, this is not a link isotopy,
but at each $t$, there will be  a
difference function, $\Delta_t$, for $L_t$
(which may have $0$ as a critical value).
 Since critical points of $\Delta _{t}$ correspond to points $\left( \left(
x_{0},y_{0},z_{1}\right) ,\left( x_{0},y_{0},z_{0}\right) \right) \in
\Lambda _{1}^{t}\times \Lambda _{0}^{t}$, $\Delta _{1}$ has no critical
points at all.  Therefore $H_{k}\left( \Delta _{1}\right) \simeq 0$ for all 
$k\in \Z$.
It remains to show that $H_{k}\left( \Delta _{0}\right) \simeq H_{k}\left(
\Delta _{1}\right) $.  For the family of difference functions $\Delta _{t}$, choose paths $\alpha ,\beta \co \left[ 0,1\right] \rightarrow \R$
such that $\alpha \left( t\right) $ is negative and less than all critical
values of $\Delta _{t}$ and $\beta \left( t\right) $ is positive and greater
than all critical values of $\Delta _{t}$.  By \fullref{lemma 3.10}, $H_{k}\left( \Delta _{0}\right) \simeq H_{k}\left( \Delta _{1}\right) $ for
all $k\in \Z$.  Therefore $H_{k}\left( \Delta _{0}\right) \simeq 0$
for all $k\in \Z$.\hfill
\end{proof}

As we will see, the positive and negative homology groups do carry
information about a particular link.
We will use polynomials $\Gamma _{\Delta}^{\pm}$ to encode information
about the set of homology groups of a difference function $\Delta $.  Define 
\begin{eqnarray*}
\Gamma _{\Delta}^{+}\left( \lambda \right) &=&\sum_{k=0}^{\infty}
\text{dim}H_{k}^{+}\left( \Delta \right) \cdot \lambda ^{k}, \\
\Gamma _{\Delta}^{-}\left( \lambda \right) &=&\sum_{k=0}^{\infty}
\text{dim}H_{k}^{-}\left( \Delta \right) \cdot \lambda ^{k}\text{.}
\end{eqnarray*}
As noted above,  if $\Delta $ and $\Delta ^{\prime}$ are two
difference functions for a link $L$, then there exists some $i\in \Z$
such that $\Gamma _{\Delta}^{\pm}\left( \lambda \right) =\Gamma _{\Delta
^{\prime}}^{\pm}\left( \lambda \right) \cdot \lambda ^{i}$.  Thus we 
define positive and negative homology polynomials for $L$ as \emph{normalized}
versions of the positive and negative homology polynomials for $\Delta $,
where $\Delta $ is some difference function for $L$.

\begin{definition}
Let $L=\left( \Lambda _{1},\Lambda _{0}\right) $ be a Legendrian link of maximal unknots  with difference function $\Delta $.  Define the \emph{positive and negative homology polynomials of} $L$ by 
\begin{eqnarray*}
\Gamma ^{+}\left( \lambda \right) \left[ L\right] &=&\Gamma _{\Delta
}^{+}\left( \lambda \right) \cdot \lambda ^{i}\text{,} \\
\Gamma ^{-}\left( \lambda \right) \left[ L\right] &=&\Gamma _{\Delta
}^{-}\left( \lambda \right) \cdot \lambda ^{i}\text{,}
\end{eqnarray*}
where $i\in \Z$ is chosen so that $\Gamma ^{-}\left( \lambda \right)
\left[ L\right] $ has degree zero.
\end{definition}

\begin{remark} \label{vector}
 An alternative to looking at a normalized version of the
polynomials is to consider a vector encoding the dimensions of the homology
groups.  For example, if $\text{dim} H_k^-(\Delta) = 0$ if $k < A$ or $k>B$
and $H_A^-(\Delta)$ and $H_B^-(\Delta)$ are nontrivial, then we construct
the \emph{negative homology vector}
$$\left( \text{dim} H_A^-(\Delta), \text{dim} H_{A+1}^-(\Delta), \dots,
\text{dim} H_B^-(\Delta)  \right).$$
Similarly, one can construct the positive homology vector.
\end{remark}

In fact, the  homology polynomials are  invariants of $L$.

\begin{theorem}
$\Gamma ^{\pm}\left( \lambda \right) \left[ L\right] $ are well-defined
invariants of a Legendrian link $L=\left( \Lambda _{1},\Lambda _{0}\right)$ of 
 maximal unknots.
\end{theorem}

\begin{proof}
We have already seen that the homology polynomials for $L$ do not depend on
the choice of difference function.  It remains to show that the polynomials
do not change as $L$ undergoes a Legendrian isotopy.

Suppose $L_{t}$, $t\in \left[ 0,1\right] ,$ is a 1--parameter family of
Legendrian links made of two maximal unknots.  By \fullref{path lifting
property}, there exists a difference function $\Delta _{t}$ for each $L_{t}$ such that $\Delta
_{0}=\Delta _{t}$ outside a compact set.  We will show that for each $t\in
\left[ 0,1\right] $, we have
\begin{eqnarray*}
H_{k}^{+}\left( \Delta _{t}\right) \simeq H_{k}^{+}\left( \Delta
_{0}\right) , \quad
H_{k}^{-}\left( \Delta _{t}\right) \simeq H_{k}^{-}\left( \Delta
_{0}\right),
\end{eqnarray*}
for all $k\in \Z$.  Choose paths $\alpha ,\beta ,\gamma \co \left[ 0,1
\right] \rightarrow \R$ such that $\alpha \left( t\right) $ is negative and
is less than all critical values of $\Delta _{t}$, $\beta \left( t\right) \equiv 0$, and $\gamma \left( t\right) $ is positive and is greater than all critical
values of $\Delta _{t}$.  Then by construction (and since $0$ is never a
critical value of $\Delta _{t}$), $\alpha \left( t\right) $, $\beta \left(
t\right) $, and $\gamma \left( t\right) $ are noncritical values of $\Delta
_{t}$ for all $t$.  Thus by \fullref{lemma 3.10}, the above result
holds.\hfill
\end{proof}

In the remainder of this section, we will prove a few facts about the
homology groups and homology polynomials for a Legendrian link made of two
maximal unknots.

The following lemma will be used to relate the positive and negative homology polynomials
for $L$ and to calculate polynomials for particular links in \fullref{poly_calcs}.
    This Lemma agrees with \cite[Lemma~3.13]{Generating}
and the proof can be found there.

\begin{lemma}
\label{lemma 3.13}For a function $\Delta \co \R \times \R^{1 + n_1} \times \R^{1+n_0} \rightarrow \R$,
and $a,b,c$ noncritical values of $\Delta $ with $a<b<c$, there is a long
exact sequence
$$\cdots \overset{\partial _{\ast}}{\longrightarrow}H_{k}\bigl( \Delta
^{b},\Delta ^{a}\bigr) \overset{i_{\ast}}{\longrightarrow}H_{k}\bigl(
\Delta ^{c},\Delta ^{a}\bigr) \overset{\pi _{\ast}}{\longrightarrow}
H_{k}\bigl( \Delta ^{c},\Delta ^{b}\bigr) \overset{\partial _{\ast}}{\longrightarrow}
H_{k-1}\bigl( \Delta ^{b},\Delta ^{a}\bigr)
\overset{i_{\ast}}{\longrightarrow}\cdots.$$
\end{lemma}

\begin{theorem} \label{positive/negative}
Let $L=\left( \Lambda _{1},\Lambda _{0}\right) $ be a Legendrian link where $\Lambda _{1}$ and $\Lambda _{0}$
are maximal unknots.  Then $\Gamma ^{+}\left( \lambda \right) \left[ L\right]
=\lambda \cdot \Gamma ^{-}\left( \lambda \right) \left[ L\right] $.
\end{theorem}

\begin{remark} Notice that in the vector notation described in
\fullref{vector}, this says that our links will have the
same negative and positive homology vectors.
Because of this dependence of $\Gamma^+(\lambda)[L]$
on $\Gamma^-(\lambda)[L]$, in the following result statements, usually only 
$\Gamma^-(\lambda)[L]$ will be discussed.
\end{remark}

\begin{proof}
Let $\Delta $ be a difference function for $L$.  Then by \fullref{proposition 3.12}, $H_{k}\left( \Delta \right) \simeq 0$ for all $k\in 
\Z$.  Choose $m\in \R$ large enough so that all the
critical values of $\Delta $ are within the interval $\left[ -m+\epsilon
,m-\epsilon \right] $ for some $\epsilon >0$, and recall that $0$ is a
noncritical value of $\Delta $.  Therefore by \fullref{lemma 3.13} we
have the following exact sequence, for any $k\in \Z$.
\begin{equation*}
\cdots 
\overset{i_{\ast}}{\longrightarrow}H_{k}\bigl(
\Delta ^{m},\Delta ^{-m}\bigr) \overset{\pi _{\ast}}{\longrightarrow}
H_{k}\bigl( \Delta ^{m},\Delta ^{0}\bigr) \overset{\partial
_{\ast}}{\longrightarrow}H_{k-1}\bigl( \Delta ^{0},\Delta ^{-m}\bigr)
\overset{i_{\ast}}{\longrightarrow}\cdots.
\end{equation*}
Note that
\vadjust{\vskip-10pt}
\begin{align*}
H_{k}(\Delta^{m},\Delta^{-m}) &\simeq H_{k}(\Delta) \simeq 0,\\
H_{k}(\Delta^{m},\Delta ^{0}) &\simeq H_{k}^{+}(\Delta),\\
\text{and}\quad
H_{k}(\Delta^{0},\Delta^{-m}) &\simeq H_{k}^{-}(\Delta).
\end{align*}
Thus the above
sequence can be rewritten as
\begin{equation*}
\cdots 
\overset{i_{\ast}}{\longrightarrow}0\overset{\pi _{\ast}}{\longrightarrow}H_{k}^{+}\left(
\Delta \right) \overset{\partial _{\ast}}{\longrightarrow}
H_{k-1}^{-}\left( \Delta \right) \overset{i_{\ast}}{\longrightarrow}0
\overset{\pi _{\ast}}{\longrightarrow}\cdots.
\end{equation*}
Hence the $\partial _{\ast}$ maps are all isomorphisms, which tells us that for all $k\in \Z$, $H_{k+1}^{+}\left( \Delta \right) \simeq H_{k}^{-}\left(
\Delta \right) $. 
\end{proof}

Finally, we will end the section by showing how the
negative homology polynomials can sometimes detect if a link is ``ordered".
  We say a link $L = (\Lambda_1, \Lambda_0)$ is \emph{ordered} if it is not
Legendrian equivalent to $\overline L = (\Lambda_0, \Lambda_1)$.   

\begin{theorem} \label{order}
Let $L=\left( \Lambda _{1},\Lambda _{0}\right) $ be a Legendrian link  of maximal unknots.  If there does not exist an  $l\in \Z$ such that $\Gamma
^{-}\left( \lambda \right) \left[ L\right] =$ $\lambda ^{l}\cdot \Gamma
^{-}\left( \lambda ^{-1}\right) \left[ L\right] $, then the link $L$ is
ordered.
\end{theorem}

\begin{remark} \label{order_vector} In terms of the vector notation, the  
above corollary says that a link $L$ is ordered if the vector associated
to $\Gamma^-$ is not symmetric.
For example, it will be shown that the
rational link $L=\left( 2,1,4\right) $, which is shown on the
left side  of \fullref{fig:intro_flypes}, has $\Gamma ^{-}\left( \lambda \right) \left[ L\right]
=2\lambda ^{0}+\lambda ^{-1}$ and  so a corresponding vector of $(1, 2)$, 
and therefore it must be ordered.  However
the rational link $L^{\prime}=\left( 2,1,2\right) $, with $\Gamma
^{-}\left( \lambda \right) \left[ L^{\prime}\right] =\lambda ^{0}+\lambda
^{-1}$ or vector equal to $(1,1)$ is potentially unordered.
\end{remark}

\begin{proof}  Let $L = (\Lambda_1, \Lambda_0)$, 
$\overline L = (\Lambda_0, \Lambda_1)$.  
We will show that   there exists an integer $l\in \Z$ such that $\Gamma ^{-}\left( \lambda \right) \left[ L\right] =\lambda ^{l}\cdot \Gamma
^{-}\left( \lambda ^{-1}\right) \bigl[ \overline{L}\bigr] $.
Let $F_{1},F_{0}$ be LQ generating families for $\Lambda _{1}$ and $\Lambda
_{0}$, respectively.  Then $\Delta =F_{1}-F_{0}$ is a difference function
for $L$ and $\overline{\Delta}=F_{0}-F_{1}=-\Delta $ is a difference
function for $\overline{L}$.   Then if $N$ denotes the dimension of
the domain of $\Delta$, we have 
\begin{align*}H_k^+(\Delta) = H_{k}(\Delta^{+\infty}, \Delta^0) &\simeq H^{N-k}(\Delta^{+\infty}, \Delta^0)\\
&\simeq  H_{N- k}\bigl( \overline{\Delta}^0,
\overline{\Delta}^{-\infty}\bigr) = H_{N- k}^-(\overline \Delta). \end{align*}
Therefore the homology polynomials of $\Delta $ and the homology polynomials
of $\overline{\Delta}$ are related as follows:
$$\Gamma _{\Delta}^{-}\bigl( \lambda \bigr) = 
\lambda ^{N}\cdot \Gamma _{\overline{\Delta}}^{+}\bigl( \lambda^{-1}\bigr)$$
Since the homology polynomials of $L$ and $\overline{L}$ are normalized
versions of the polynomials of $\Delta $ and $\overline{\Delta}$, there
exists an integer $l^\prime$ such that
$$ \Gamma ^{-}\left( \lambda \right) [L]  = 
\lambda ^{l^\prime}\cdot \Gamma ^{+}\bigl( \lambda ^{-1}\bigr)
\bigl[\overline{L}\bigr].$$
By \fullref{positive/negative},
$  \Gamma
^{+}\left( \lambda ^{-1}\right) \bigl[ \overline L\bigr] =\lambda ^{-1}\cdot \Gamma
^{-}\left( \lambda ^{-1}\right) \bigl[ \overline L\bigr] $.
Therefore 
$$
 \Gamma ^{-}\left( \lambda \right) \bigl[ L\bigr]  = \lambda
^{l^\prime}\cdot \Gamma ^{+}\bigl( \lambda ^{-1}\bigr) \bigl[ 
\overline{L}\bigr] =\lambda ^{l^\prime}\cdot \lambda ^{-1}\cdot \Gamma
^{-}\bigl( \lambda ^{-1}\bigr) \bigl[ \overline{L}\bigr].\proved
$$
\end{proof}

\section{Polynomial calculations} \label{poly_calcs}
In order to calculate the homology polynomials for a given Legendrian link $L=\left( \Lambda _{1},\Lambda _{0}\right) $,  it will sometimes be sufficient
to know the nondegenerate
critical points of a difference function for $L$, as well as the critical
values and indices of the critical points.  Recall that a critical point
corresponds to a point $\left( \left( x_{0},y_{0},z_{1}\right) ,\left(
x_{0},y_{0},z_{0}\right) \right) \in \Lambda _{1}\times \Lambda _{0}$, and
the critical value is equal to $z_{1}-z_{0}$.  Thus we can find all
critical points and their values from looking at the front projection of $L$
, without finding the difference function $\Delta $ explicitly.  In this
section we explain how we can also recover the index of a critical point, up
to a shift, from the front projection of $L$.  The definitions and
propositions in this section are based on those found in 
\cite[Section~5]{Generating}.

We begin with some definitions for use with the front projection of a knot
or link.  Given a Legendrian knot $\Lambda \subset  \jetR $, let $\pi _{xz}\left( \Lambda \right) $ be the front
projection of $\Lambda $. 
Let $C$ be
the set of points in $\Lambda $ whose image under $\pi _{xz}$ is a cusp
point.  We define the \emph{branches} of $\Lambda $ to be the connected
components of $\Lambda \setminus C$.  Branches $B_{0}$, $B_{1}$ are said to
be \emph{adjacent} if their closures, $\overline{B_{0}}$ and $\overline{B_{1}}$, intersect.  Given two adjacent branches $B_{0}$ and $B_{1}$, we say $B_{1}>B_{0}$ if there exists some $b\in \overline{B_{0}}$ $\cap \overline{B_{1}}$ and a path $\gamma \co \left[ 0,1\right] \rightarrow \pi _{xz}\left(
\Lambda \right) $ such that $\gamma \bigl[ 0,\frac{1}{2}\bigr)\subset \pi
_{xz}\left( B_{0}\right) $, $\gamma \bigl(\frac{1}{2},0\bigr]\subset \pi _{xz}\left(
B_{1}\right) $, and $\gamma \bigl( \frac{1}{2}\bigr) =\pi _{xz}\left(
b\right) $ where $\pi _{xz}\left( b\right) $ is an up-cusp along the path.

Now we will describe a way of assigning integers, called branch indices, to
each branch of a knot.   Let $\Lambda \subset  \jetR $ be a maximal unknot, and choose $p_{0}$  to be a marked
point of $\Lambda $ such that $\pi _{xz}\left( p_{0}\right) $ is not a cusp
point of $\pi _{xz}\left( \Lambda \right) $.  Let $\left\{ B_{i}\right\} $
be the set of branches of $\Lambda $ such that $B_{0}$ is the branch
containing $p_{0}$.  We say $B_{0}$ is the \emph{initial branch} of
$\Lambda $.  We then define the branch index $i_{Br}\co \left\{ B_{i}\right\}
\rightarrow \Z$ as follows:

\begin{enumerate}
\item[(1)] $i_{Br}\left( B_{0}\right) =0$, and

\item[(2)] $i_{Br}\left( B_{i}\right) -i_{Br}\left( B_{j}\right) =1$ if $B_{i}$, $B_{j}$ are adjacent with $B_{i}>B_{j}$.
\end{enumerate}

Note that given a marked point, the branch index is well-defined for all
branches of $\Lambda $.  Recall that for each point on a branch, there
is a corresponding critical point of the generating family restricted to a
fiber.   As explained in
\cite[Proposition~5.3]{Generating}, there is some integer so that the branch index of a point corresponds, 
up to a shift by this integer, to the index of the corresponding fiber critical point.

The final piece we need to calculate indices of critical points of a
difference function $\Delta $ is called the graph index of a critical point.

\begin{definition}
Let $\Lambda _{1},\Lambda _{0}\subset \jetR $
be Legendrian knots with branches $B_{1},B_{0}$ with $\left( x_{0},y_{0},z_{1}\right) \in B_{1}$ and $\left(
x_{0},y_{0},z_{0}\right) \in B_{0}$.  Then there exists a neighborhood $U$
of $x_{0}$ in $\R$ and functions $g_{1},g_{0}\co U\rightarrow \R$
such that near $\left( x_{0},y_{0},z_{1}\right) $, $\pi _{xz}\left(
B_{1}\right) =\{ \left( x,g_{1}\left( x\right) \right) \} $ and near $\left(
x_{0},y_{0},z_{0}\right) $, $\pi _{xz}\left( B_{0}\right) = \{ \left(
x,g_{0}\left( x\right) \right) \}$.  Note that $x_{0}$ is a critical point of 
$\Gamma =g_{1}-g_{0}$.  We say $\left( \left( x_{0},y_{0},z_{1}\right)
,\left( x_{0},y_{0},z_{0}\right) \right) \in \Lambda _{1}\times \Lambda _{0}$
is \emph{nondegenerate} if $x_{0}$ is a nondegenerate critical point of $g_{1}-g_{0}$.  The \emph{graph index} $i_{\Gamma}$ of a nondegenerate
critical point $x_{0}$ is the Morse index of $g_{1}-g_{0}$ at $x_{0}$.
\end{definition}

The following proposition allows us to calculate the indices of critical
points of a difference function $\Delta $.

\begin{proposition} \label{indices}
Let $L=\left( \Lambda _{1},\Lambda _{0}\right) $ be a Legendrian link  of maximal unknots, and let $\Delta $ be a
difference function for $L$.  Suppose  $q\in \R \times \R^{1+ n_1} \times \R^{1+n_0}$ is a 
nondegenerate critical point of $\Delta $ with 
corresponding point
 $\left( \left( x_{0},y_{0},z_{1}\right)
,\left( x_{0},y_{0},z_{0}\right) \right) \in \Lambda _{1}\times \Lambda _{0}$.
Say that $B_{1},B_{0}$ are the branches of $\Lambda _{1},\Lambda _{0}$
containing $\left( x_{0},y_{0},z_{1}\right) ,\left( x_{0},y_{0},z_{0}\right) 
$, respectively.
 Then, for any choice of initial branches for $\Lambda_1$ and $\Lambda_0$,
there is a $c \in \Z$ such that the
index of $q$  is equal to $i_{Br}\left( B_{1}\right) -i_{Br}\left( B_{0}\right)
+i_{\Gamma}\left( x_{0}\right) + c $.
\end{proposition}

We will often refer to the integer $i_{Br}\left( B_{1}\right) -i_{Br}\left( B_{0}\right)
+i_{\Gamma}\left( x_{0}\right)$ as the {\it relative index} of $q$;  this relative index
is only well-defined
up to a constant depending on the choice of initial branches.
\fullref{indices} can be proved with minor adaptions to the proof of
\cite[Proposition~5.5]{Generating}. 

The following lemma will be our main tool for calculating the homology polynomials
of particular links.  This is essentially
\cite[Proposition~4.2]{Generating}.  The proof
found there consists of studying a number of long exact sequences as given
by \fullref{lemma 3.13}.

\begin{lemma}
\label{proposition 4.2}Suppose the Legendrian link $L=\left( \Lambda
_{1},\Lambda _{0}\right) $ of
maximal unknots has a difference function $\Delta \co \R \times \R^{1 + n_1}
\times \R^{1 + n_0} \rightarrow \R$ with critical values $c_{0}^{\pm
},c_{1}^{\pm},\dots ,c_{n}^{\pm}$ and noncritical values $a_{0},a_{1},\dots ,a_{n},b_{0},b_{1},\dots ,b_{n}$ satisfying
\begin{equation*}
a_{0}<c_{0}^{-}<a_{1}<c_{1}^{-}<\dots
<a_{n}<c_{n}^{-}<0<c_{n}^{+}<b_{n}<\dots <c_{1}^{+}<b_{1}<c_{0}^{+}<b_{0}\text{.}
\end{equation*}
If

\begin{enumerate}
\item[(1)] For $j=0,1,\dots ,n$, there exist $w_{j}$ nondegenerate critical
points with value $c_{j}^{+}$ and $w_{j}$ nondegenerate critical points with
value $c_{j}^{-}$;

\item[(2)] For a given labeling of branch indices, all critical points of
value $c_{j}^{+}$ have relative index $i_{j}+1$ and all critical points of
value $c_{j}^{-}$ have relative index $i_{j}$;

\item[(3)] For $j=1,2,\dots ,n$, $H_{\ast}\left( \Delta ^{b_{j}},\Delta
^{a_{j}}\right) =0$ for all $\ast \in \Z$;
\end{enumerate}

Then
$$
\Gamma ^{-}\left( \lambda \right) \left[ L\right] =\lambda ^{h}\cdot
\sum_{j=0}^{n}w_{j}\lambda ^{i_{j}},
$$
where $h$ is chosen so that $\Gamma ^{-}\left( \lambda \right) \left[ L
\right] $ has degree $0$.
\end{lemma}

We can now state and prove formulas for  the homology
polynomials for several types of links.  The first links we
consider are the rational links $\left( 2w_{n},k_{n},2w_{n-1},\dots
,k_{1},2w_{0}\right) $ as described in \fullref{intro}.
\fullref{theorem 6.1} is similar to \cite[Theorem~6.1]{Generating}, and the proof follows the same format.

\begin{theorem}
\label{theorem 6.1}Let $L=\left( 2w_{n},k_{n},2w_{n-1},\dots
,k_{1},2w_{0}\right) $ (see \fullref{fig:rational_link}).  Then 
$$
\Gamma^{-}\left( \lambda \right) \left[ L\right] = w_{0}\lambda
^{0}+w_{1}\lambda ^{-k_{1}}+w_{2}\lambda ^{-\left( k_{1}+k_{2}\right)
}+\dots +w_{n}\lambda ^{-\left( k_{1}+\dots +k_{n}\right)}.
$$
\end{theorem}

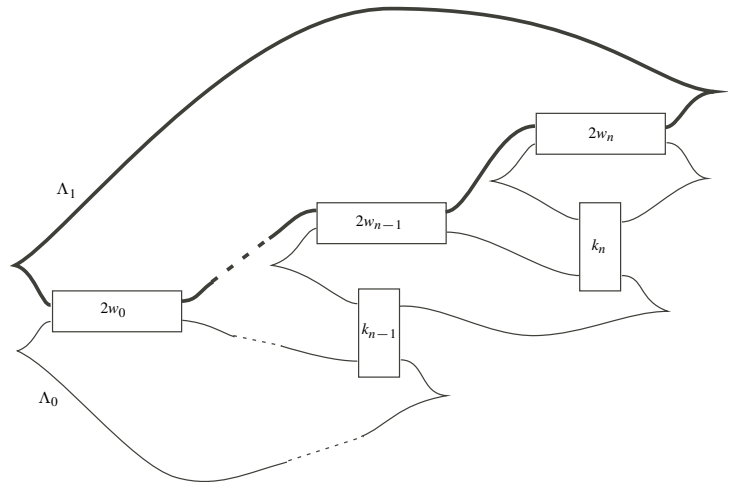
\begin{figure}[ht]
\labellist\tiny
\pinlabel {$\Lambda_1$} [br] at 60 230
\pinlabel {$\Lambda_0$} [tr] at 45 80
\pinlabel {$2w_0$} at 85 143
\pinlabel {$k_{n{-}1}$} at 303 125
\pinlabel {$2w_{n-1}$} at 303 214
\pinlabel {$k_n$} at 483 195
\pinlabel {$2w_n$} at 485 287
\endlabellist
\centerline{\includegraphics[height=2.5in]{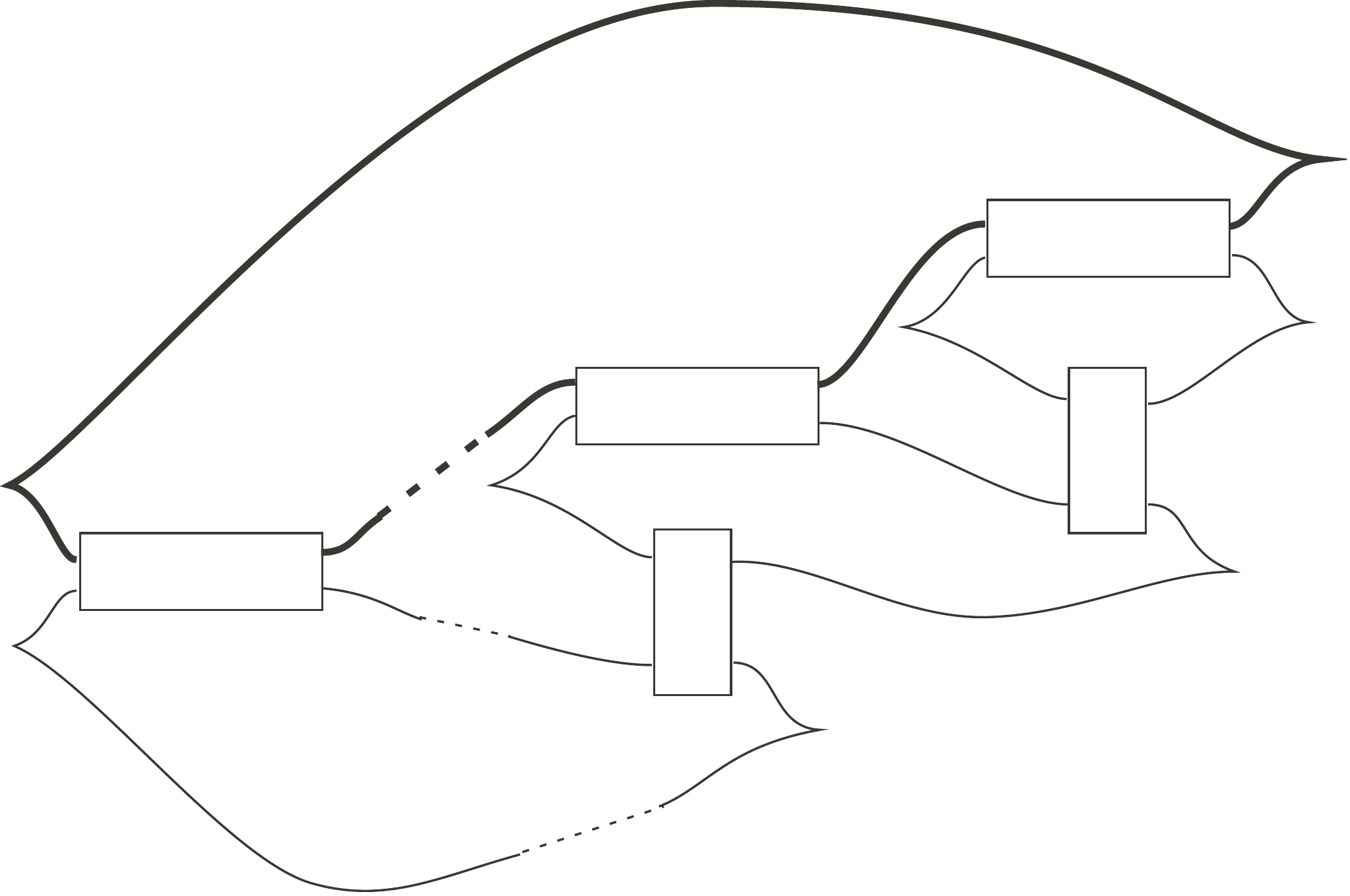}}
\caption{The link $L=\left(2w_{n},k_{n},\dots ,2w_{1},k_{1},2w_{0}\right)$}
        \label{fig:rational_link}
\end{figure} 

\begin{remark} \label{rat_vector}
Using the vector notation, this says that the negative homology
polynomial of $L=\left( 2w_{n},k_{n},2w_{n-1},\dots
,k_{1},2w_{0}\right) $  is given by the vector
$$(w_n, 0, \dots, 0, w_{n-1}, 0, \dots,0, w_0)$$
where there are precisely $(k_j-1)$ zeros between $w_j$ and $w_{j-1}$.
\end{remark}

\begin{proof}
It is possible to isotope $L$ so that it has a difference function 
$\Delta \co \R \times \R^{1 + n_1} \times \R^{1 + n_0}  \rightarrow \R$ with $2\left( w_{0}+w_{1}+\dots
+w_{n}\right) $ nondegenerate critical points.  In particular, we can
choose the branch indices such that, for $i=0,1,\dots ,n$, we have the
following:  $\Delta $ has $w_{i}$ critical points with critical value $c_{i}^{+}>0$ and relative index $1$ if $i=0$, or relative index $1-\sum_{j=1}^{i}k_{j}$
otherwise, and $w_{i}$ critical points with critical value $c_{i}^{-}<0$ and
relative index $0$ if $i=0$, or relative index $-\sum_{j=1}^{i}k_{j}$ otherwise. 
Furthermore the critical values are related as follows:
\begin{equation*}
c_{0}^{-}<c_{1}^{-}<\dots <c_{n}^{-}<0<c_{n}^{+}<\dots <c_{1}^{+}<c_{0}^{+}
\end{equation*}
\fullref{fig:rat_calc} illustrates one such construction for the link $\left( 2,1,4\right) $.

\begin{figure}[ht]
\labellist\small
\pinlabel {$c_0^+$} [b] at 92 582
\pinlabel {$c_0^-$} [t] at 144 582
\pinlabel {$c_0^+$} [b] at 179 582
\pinlabel {$c_0^-$} [t] at 232 582
\pinlabel {$c_1^+$} [b] at 314 582
\pinlabel {$c_1^-$} [t] at 377 582
\pinlabel {$0$} [t] at 372 403
\pinlabel {$1$} [t] at 548 582
\endlabellist
\centerline{\includegraphics[height=1.8in]{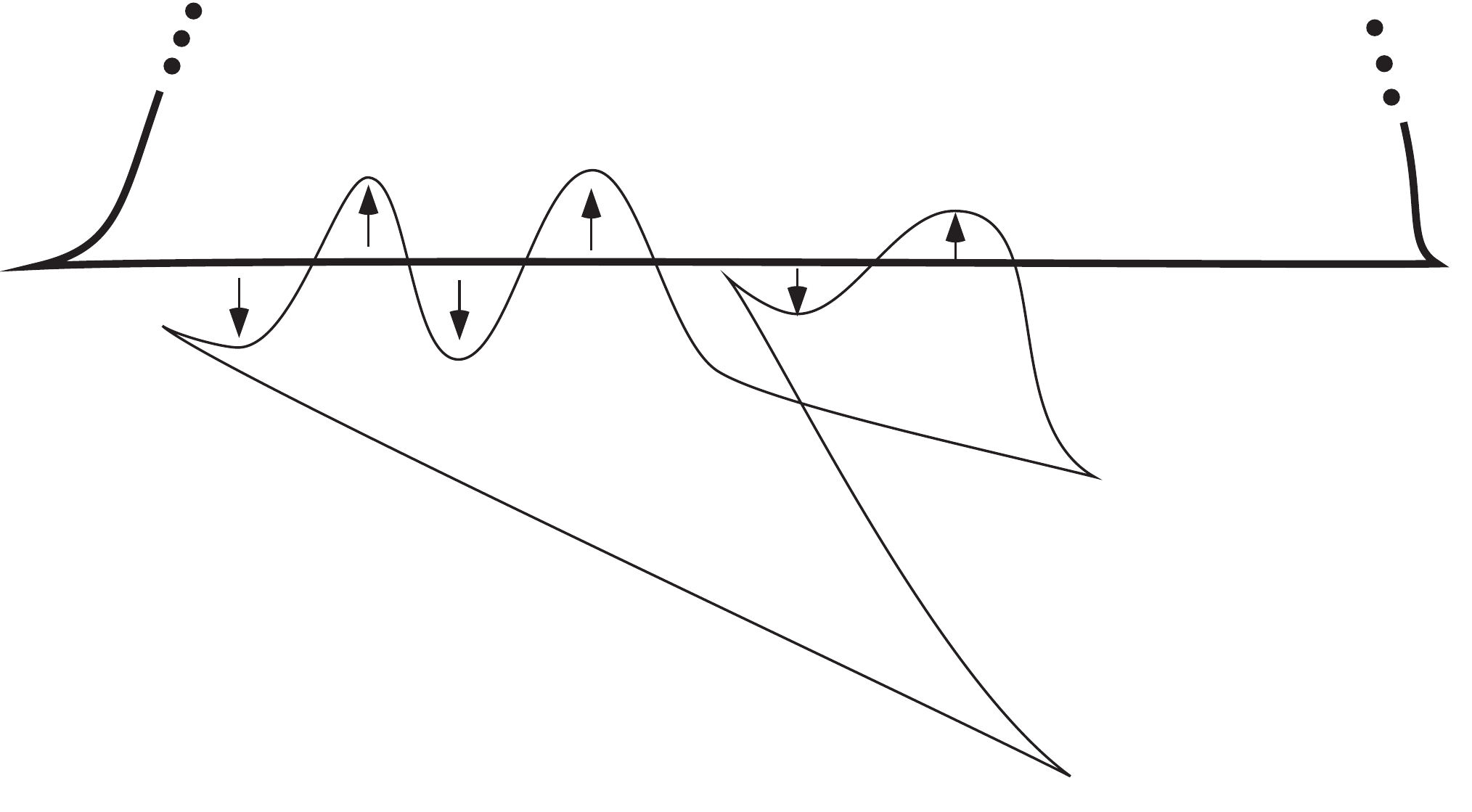}}
\caption{A portion of the Legendrian link $(2,1,4)$; the upper branch of
$\Lambda_1$ is completed to have sufficiently positive slope on the left
side so that its
difference function $\Delta $ has $2\left( 2 + 1 \right) $ nondegenerate
critical points, represented by pairs of points $\left( \left(
x_{0}, 0 ,z_{1}\right) ,\left( x_{0}, 0 ,z_{0}\right) \right) \in
\Lambda _{1}\times \Lambda _{0}$.  In this figure, each pair is joined by a directed
line segment, and the label ($c_{0}^{\pm}$ or $c_{1}^{\pm}$) near  each
pair denotes the critical value.  An initializing choice  of branch labels is indicated.}
        \label{fig:rat_calc}
\end{figure}
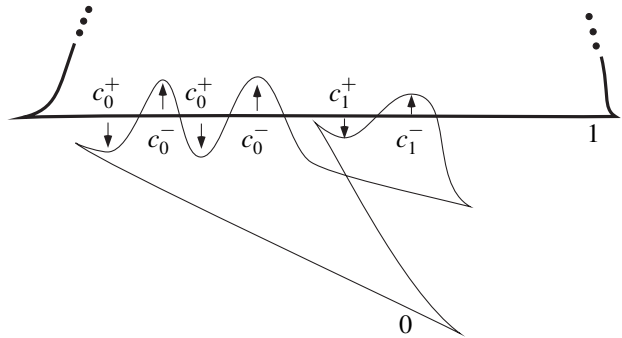

Now for all $j$, $1\leq j\leq n$, choose noncritical values $a_{j}$ and $b_{j}$ such that 
\begin{equation*}
c_{j-1}^{-}<a_{j}<c_{j}^{-}<0<c_{j}^{+}<b_{j}<c_{j-1}^{+}\text{;}
\end{equation*}
for $j = 0$, we choose noncritical values $a_0, b_0$ satisfying
$$a_0 < c_0^- < 0 < c_0^+ < b_0.$$
 For all $j$, we may apply a
deformation argument as in the proof of \fullref{proposition 3.12}
to construct a 1--parameter family of functions $\Delta
_{t}$ such that $a_{j},b_{j}$ are noncritical values of $\Delta _{t}$ for
all $t\in \left[ 0,1\right] $, $\Delta _{0}=\Delta $, and $\Delta _{1}$ has
no critical points with values in $[ a_{j},b_{j}] $; this isotopy is
not an isotopy of links but rather eliminates the $2 w_n + \dots + 2w_j$ crossings.
 Thus by
\fullref{lemma 3.10}, $H_{\ast}\bigl( \Delta _{0}^{b_{j}},\Delta
_{0}^{a_{j}}\bigr) \simeq H_{\ast}\bigl( \Delta _{1}^{b_{j}},\Delta
_{1}^{a_{j}}\bigr) =0$ for all $\ast \in \Z$.  The theorem now
follows from \fullref{proposition 4.2}.
\end{proof}

\fullref{theorem 6.1} (\fullref{rat_vector}) together with \fullref{order}
(\fullref{order_vector})  then prove 
\begin{corollary}  \label{rat_order} Let $L=\left( \Lambda _{1},\Lambda _{0}\right) $ be the rational Legendrian
link
$$\left( 2w_{n},k_{n},\dots ,2w_{1},k_{1},2w_{0}\right).$$
If the
vector $\left( 2w_{n},k_{n},\dots ,2w_{1},k_{1},2w_{0}\right) $ is not
palindromic, then the link $L$ is ordered.
 \end{corollary}

Homology polynomials  for ``flypes'' of links will also have a nice
formulation.  A vertical or horizontal flype is a move wherein a portion
of a link is rotated 180$^{\circ}$ about a vertical or horizontal axis
(see \fullref{fig:flypes}).    For more background on flypes see,
for example, Adams \cite{Adams}, Conway \cite{Conway} or Traynor
\cite{Generating}.  It is known that flypes produce topologically
equivalent links, but we will use the polynomials to show that these
flypes can produce nonequivalent Legendrian links.

\begin{figure}[ht]
\labellist\small
\pinlabel {(a)} at 70 286
\pinlabel {(b)} at 526 286
\pinlabel {(c)} at 70 54
\pinlabel {(d)} at 526 54
\pinlabel {F} at -16 425
\pinlabel {\reflectbox{F}} at 186 348
\pinlabel {F} at 374 379
\pinlabel {\rotatebox{180}{\reflectbox{F}}} at 675 393
\pinlabel {F} at -22 193
\pinlabel {\reflectbox{F}} at 200 123
\pinlabel {F} at 374 156
\pinlabel {\rotatebox{180}{\reflectbox{F}}} at 676 154
\endlabellist
\centerline{\includegraphics[width=3.8in]{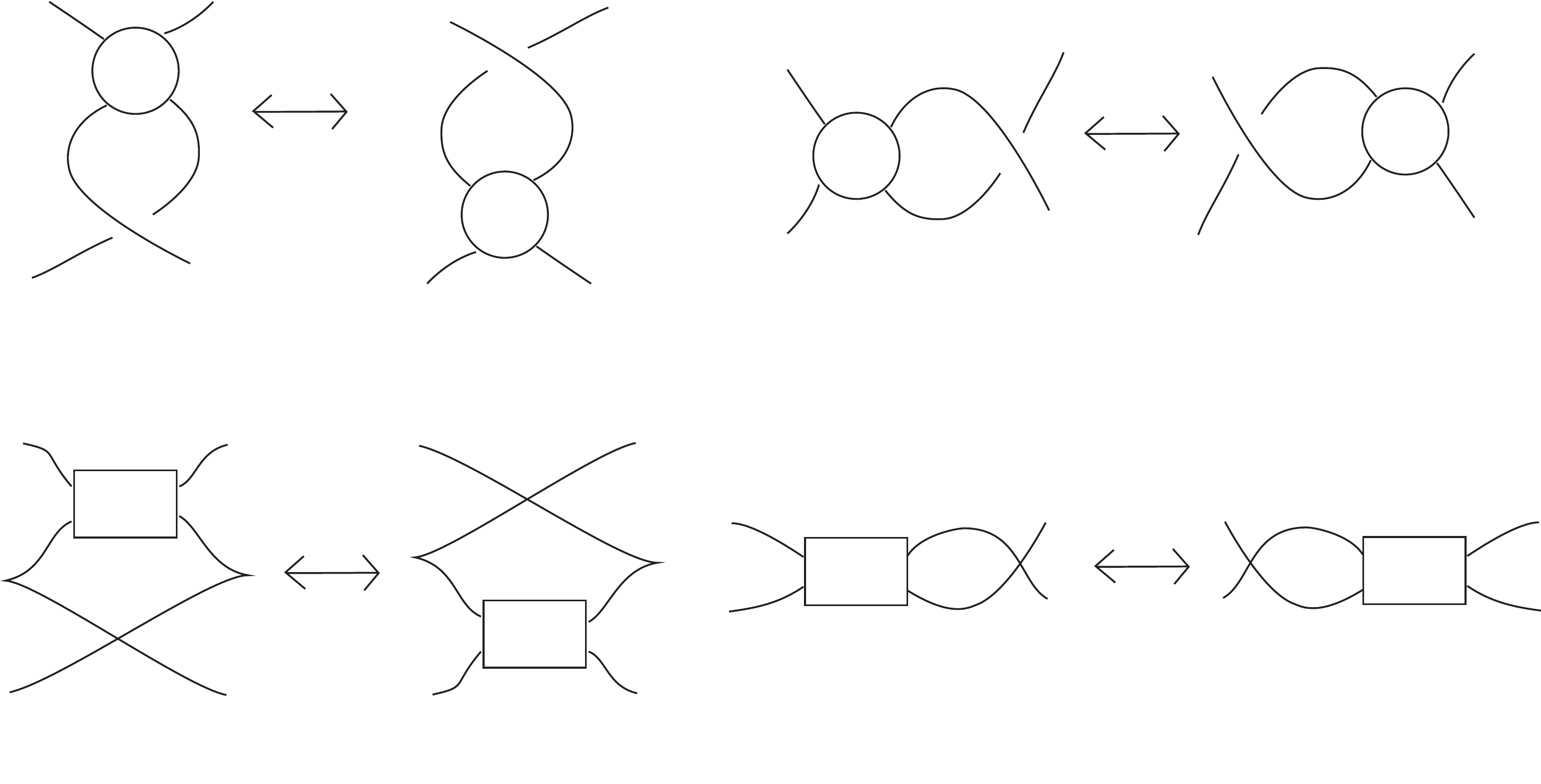}}
\caption{(a) a topological vertical flype, (b) a topological horizontal flype, (c) a Legendrian vertical flype, (d) a Legendrian horizontal flype}
\label{fig:flypes}
\end{figure}
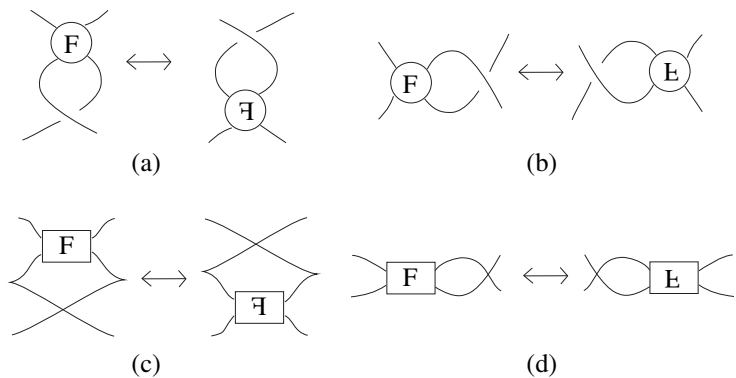

Rational links give us many opportunities to apply flypes. As mentioned in
the Introduction, we will only consider horizontal flypes since vertical
flypes produce equivalent links.   Given the rational link 
$L=\left( 2w_{n},k_{n},2w_{n-1},\dots
,k_{1},2w_{0}\right) $, we may apply one or more horizontal flypes using the
horizontal crossings represented by the terms $2w_{0},2w_{1},\dots ,2w_{n-1}$
.  Let $\left( 2w_{n},k_{n},2w_{n-1}^{p_{n-1}},\dots
,k_{1},2w_{0}^{p_{0}}\right) $ represent $L$ after it has undergone $p_{i}$
horizontal flypes using the set of $2w_{i}$ horizontal crossings of $L$,
where $0\leq i\leq n-1$.  See, for example, \fullref{fig:flypes_ex},
where the boxed portion of the link is the part rotated in the flype. 

\begin{figure}[ht]
\centerline{		\includegraphics[width=4in]{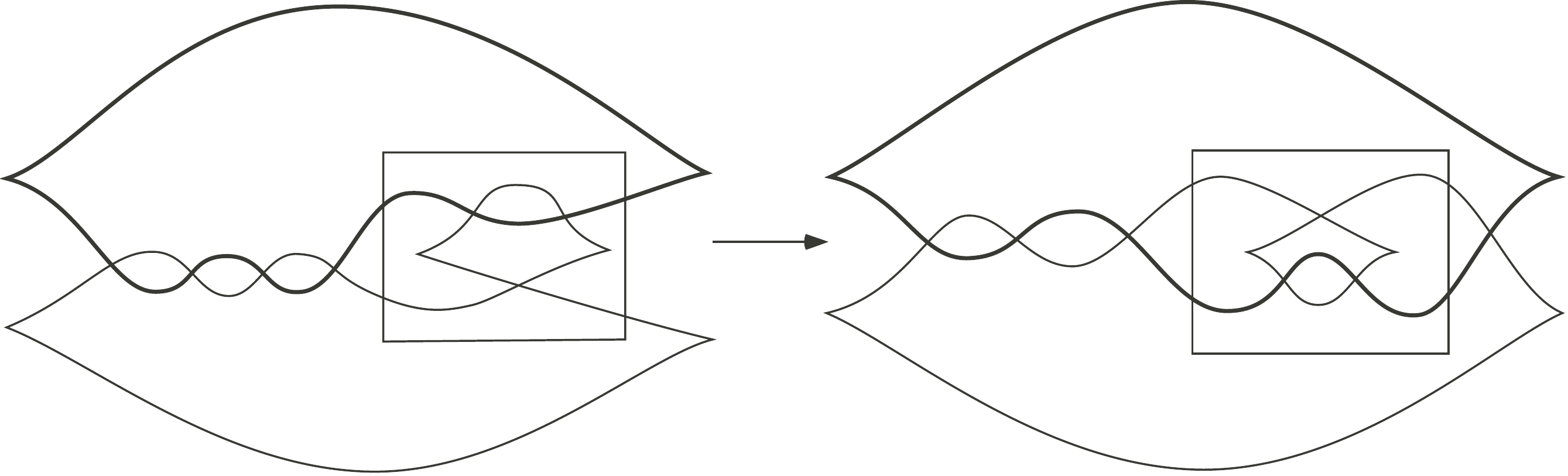}}
\caption{The rational link $\left( 2,1,4\right)$ undergoes a horizontal flype and becomes $\left( 2,1,4^{1}\right) $.  These links have different negative homology polynomials.}
        \label{fig:flypes_ex}
\end{figure}

\begin{theorem}
\label{flypesthm}Let $L=\left( 2w_{n},k_{n},2w_{n-1}^{p_{n-1}},\dots
,k_{1},2w_{0}^{p_{0}}\right) $.  For $j=0,1,\dots ,n-1$, let $\sigma \left(
j\right) =1+\sum_{i=0}^{j}p_{i}$ mod $2$.  Then 
\begin{equation*}
\Gamma ^{-}\left( \lambda \right) \left[ L\right] =\lambda ^{m}\cdot \Bigl[
w_{0}\lambda ^{0}+\sum_{i=1}^{n}w_{i}\lambda ^{[ \left( -1\right)
^{\sigma \left( 0\right)}k_{1}+\cdots +\left( -1\right) ^{\sigma \left(
i-1\right)}k_{i}]}\Bigr] \text{,}
\end{equation*}
where $m$ is chosen so that $\Gamma ^{-}$ has degree zero.
\end{theorem}

\begin{remark}  To construct the vector corresponding to the polynomial
for $L=\left( 2w_{n},k_{n},2w_{n-1}^{p_{n-1}},\dots
,k_{1},2w_{0}^{p_{0}}\right) $, one starts by writing a $w_0$ and then one
goes backward (if $p_0$ is even) or forward (if $p_0$ is odd) $k_1$ places 
to write $w_1$.  Then one repeats this procedure continuing in the same
forward/backward direction as in the previous step 
 if $p_1$ is even while changing direction if $p_1$ is odd, and
adding the entries if one arrives at a position in the vector already visited.
For example,  if $L=\left( 2,1,4\right) $ then the negative homology vector
is $(1,2)$ (or, equivalently, $\Gamma ^{-}\left( \lambda \right) \left[ L
\right] =2\lambda ^{0}+\lambda ^{-1}$); after applying 
 one horizontal
flype, we have the link $L^{\prime}=\left( 2,1,4^{1}\right) $ which
has negative homology vector $(2,1)$ (equivalently
$\Gamma ^{-}\left( \lambda \right) \left[ L^{\prime}\right]
=\lambda ^{0}+2\lambda ^{-1}$).   So even though $L$ and $L^\prime$
are topologically equivalent and have the same classical Legendrian invariants,
they are distinct Legendrian links; see \fullref{fig:flypes_ex}.
\end{remark}

\begin{proof}
This proof follows the proof of \cite[Theorem~6.2]{Generating}.  We will
give a sketch of the proof here.  The key is to show the existence of a
difference function $\Delta $ for $L$ such that $\Delta $ has $2\left(
w_{0}+w_{1}+\dots +w_{n}\right) $ nondegenerate critical points, and for $0\leq i\leq n$ each of the $2w_{i}$ critical points correspond to a pair of points $\left( \left( x_{0},y_{0},z_{1}\right) ,\left( x_{0},y_{0},z_{0}\right)
\right) \in \Lambda _{1}\times \Lambda _{0}$ on branches $W_{1}^{i}\subset
\Lambda _{1}$, $W_{0}^{i}\subset \Lambda _{0}$ where

\begin{enumerate}
\item the distance function $d_{i}\co  \R \rightarrow \R$ given by $d_{i}\left( x\right) =\left| z_{1}\left( x\right) -z_{0}\left( x\right)
\right| $, where $\bigl( x,y_{j}^{i}\left( x\right) ,z_{j}^{i}\left(
x\right) \bigr) \in \Lambda _{j}$ is a point on branch $W_{j}^{i}$ for $j=0,1$, achieves a relative maximum at $x_{0}$ for each $i$ and

\item 
$i_{Br}\bigl( W_{1}^{i}\bigr) -i_{Br}\bigl( W_{0}^{i}\bigr) =\left\{ 
\begin{array}{ll}
0 & \text{if }i=0\text{,} \\ 
\left( -1\right) ^{\sigma \left( 0\right)}k_{1}+\dots +\left( -1\right)
^{\sigma \left( i-1\right)}k_{i} & \text{if }1\leq i\leq n\text{.}
\end{array}
\right.$
\end{enumerate}

The existence of such a function is proven by induction, using the proof of
\fullref{theorem 6.1} to prove the base case.
Once we've proved the existence of a suitable difference function $\Delta $,
the rest of the proof follows using the same reasoning as in the proof of
\fullref{theorem 6.1}.\hfill
\end{proof}

\begin{theorem}
\label{Ljkpolys}Let $L_{j,k}$ be a twist link as described   in \fullref{fig:Ljk}.  Then
$$
\Gamma ^{-}\left( \lambda \right) \left[ L_{j,k}\right] = \lambda
^{0}+\lambda ^{-\left| j-k\right|}.
$$
\end{theorem}

\begin{remark} When $j \neq k$, the homology vector of $L_{j,k}$ is
$(1, 0, \dots, 0, 1)$ where there are $|j-k|-1$ zeros, while if $j=k$, the
homology vector is $(2)$.  
\end{remark}

\begin{proof}
It is possible to isotope $L_{j,k}$ so that it has a difference function $\Delta \co  \R\times \R^{1+n_1} \times \R^{1+n_0}
\rightarrow \R$ with four
nondegenerate critical points, as follows.  The critical values of the
critical points are $c_{1}^{+},c_{1}^{-}$ and $c_{0}^{+}, c_{0}^{-}$ with
the relative indices of the critical points (with an appropriate labeling of
branch indices) $1,0$ and $j-k+1, j-k$, 
respectively.  Moreover the critical values satisfy the
inequalities
\begin{equation*}
c_{0}^{-}<c_{1}^{-}<0<c_{1}^{+}<c_{0}^{+}.
\end{equation*}
\fullref{fig:twlinks} illustrates one such construction for the link $L_{2,3}$.

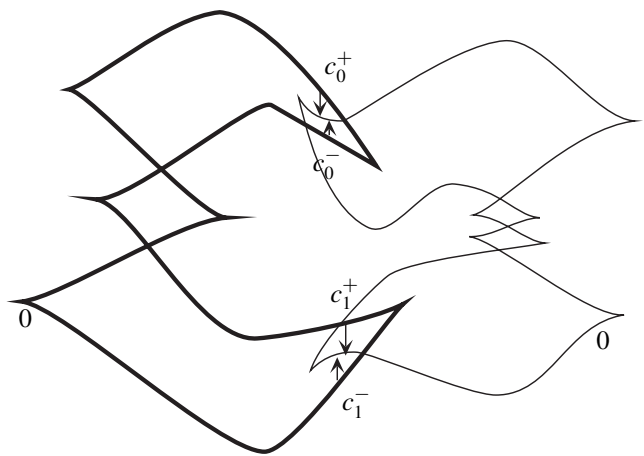
\begin{figure}[ht]
\labellist\small
\pinlabel {$c_0^+$} [bl] at 283 421
\pinlabel {$c_0^-$} [t] at 289 384
\pinlabel {$c_1^+$} [b] at 301 262
\pinlabel {$c_1^-$} [tl] at 295 217
\pinlabel {$0$} [t] at 80 273
\pinlabel {$0$} [tl] at 471 256
\endlabellist
\centerline{\includegraphics[height=2.5in]{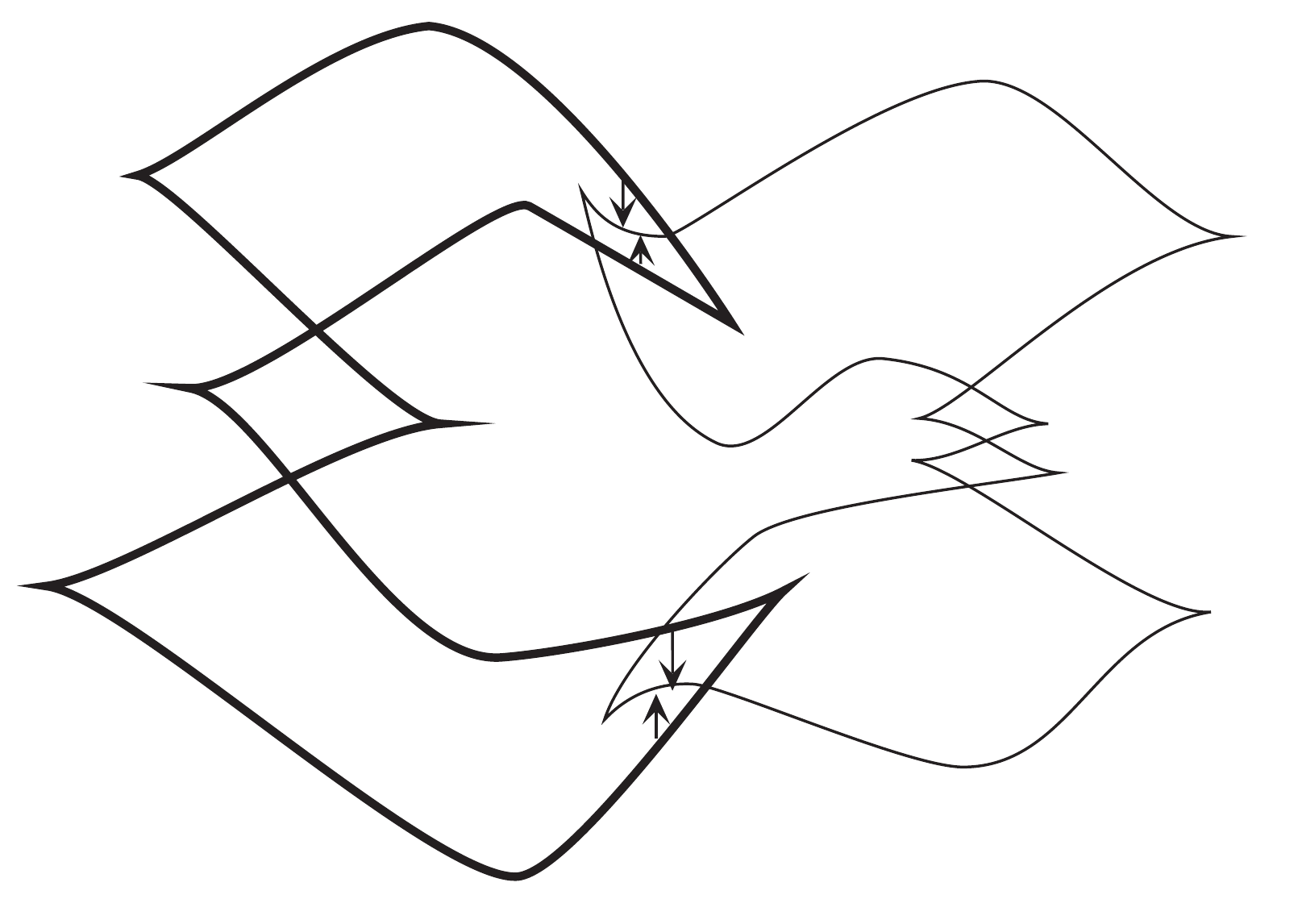}}
\caption{The Legendrian link $L_{2,3}$ can be positioned so that 
its difference function $\Delta $ has
four nondegenerate critical points, represented by pairs of points $\left(
\left( x_{0},y_{0},z_{1}\right) ,\left( x_{0},y_{0},z_{0}\right) \right) \in
\Lambda _{1}\times \Lambda _{0}$.  In this figure, each pair is joined by a
line segment and the label ($c_{0}^{\pm}$ or $c_{1}^{\pm}$) next to each
pair indicates the critical value.  An initializing choice of  branch indices is indicated.}
        \label{fig:twlinks}
\end{figure}

Choose noncritical values $a_{i},b_{i}$ of $\Delta $, $i=0,1$, such that 
\begin{equation*}
a_{0}<c_{0}^{-}<a_{1}<c_{1}^{-}<0<c_{1}^{+}<b_{1}<c_{0}^{+}<b_{0}\text{.}
\end{equation*}
As in \fullref{theorem 6.1}, we may deform $\Delta =\Delta _{0}$ (by
pulling apart the bottom portions of the knots) so that $a_{1},b_{1}$ are
never critical values of $\Delta _{t}$, and $\Delta _{1}$ has no critical
values in $\left[ a_{1},b_{1}\right] $.  Again using \fullref{lemma 3.10},
we conclude that $H_{\ast}\bigl( \Delta _{0}^{b_{1}},\Delta
_{0}^{a_{1}}\bigr) \simeq H_{\ast}\bigl( \Delta _{1}^{b_{1}},\Delta
_{1}^{a_{1}}\bigr) =0$ for all $\ast \in \Z$. Also notice that
 $H_{\ast}\bigl( \Delta
^{b_{0}},\Delta ^{a_{0}}\bigr) =H_{\ast}\bigl( \Delta \bigr) =0$ by \fullref{proposition 3.12}. 
 So by
\fullref{proposition 4.2}, 
$$
\Gamma ^{-}(\lambda) [ L_{j,k}] = \bigl( \lambda
^{0}+\lambda ^{\left( j-k\right)}\bigr) \cdot \lambda ^{h},
$$
where $h$ is chosen so that $\Gamma ^{-}( \lambda ) [L_{j,k}]$
has degree zero.  The result follows. 
\end{proof}

It is easy to see that $j+k$ determines the topological type of the link
$L_{j,k}$.  \fullref{Ljkpolys} then implies

\begin{corollary}
\label{Ljktype}Let $m\geq 2$.  If $m$ is even, the following links are
topologically but not Legendrianly equivalent:  $L_{1,m-1},L_{2,m-2},\dots
,L_{\frac{m}{2},\frac{m}{2}}$.  If $m$ is odd, the following links are not
Legendrianly equivalent:  $L_{1,m-1},L_{2,m-2},\dots ,$ $L_{\frac{m-1}{2},
\frac{m+1}{2}}$.
\end{corollary}

\section{Comparison to other results}

In this section, we first compare the results of the homology polynomial
calculations with  various decomposition number invariants
and then  with a polynomial invariant that comes from a link's
differential,
graded algebra (DGA).  

\subsection{Decomposition invariants}  There is a close relationship between
generating families and the decomposition invariants of Chekanov \cite{Chekanov2}, and 
Chekanov and Pushkar \cite{Chekanov+Pushkar}.  Briefly, an admissible 
decomposition of a front projection of a Legendrian link is a choice 
of the crossings in the front projection so that the front obtained by resolving
all the crossings is a union of two cusped ``eyes" that satisfy a set of conditions.
For a precise definition, see \cite{Chekanov2}. In fact, as pointed out
by Chekanov and Pushkar in \cite[Section 12]{Chekanov+Pushkar}, the definition of an 
admissible decomposition  
is obtained by axiomatizing combinatorial structures
arising on the front of a Legendrian submanifold defined by a generating family.

Chekanov showed that there are a number of decomposition
 invariants that can be associated to a Legendrian
link $L$.  For example, one can count the total number of admissible decompositions of
any front of $L$.
 Also, one can count the total number
of Maslov (also known as graded) admissible decompositions:  one can associate a Maslov
index to each branch of the front projection and then restrict to decompositions arising from
resolving crossings of branches having the same Maslov index. One can also look at 
associated decomposition and Maslov decomposition polynomials:
the decomposition polynomial is 
$D\left(L_{j,k}\right) = \sum_\rho z^{j(\rho)}$, where the sum is taken over all allowable decompositions $\rho$
of a front of $L_{j,k}$, and 
$j(\rho)$ is the number of left cusps  
in the decomposition
minus the number of  crossings resolved in the decomposition.
 The Maslov indices for the branches depend on a choice of initial value for a
 branch of each component, and different choices for the
 Maslov branch indices lead to different Maslov crossings for a $2$ or more component link.  
Hence the
Maslov decomposition number of a multi-component link is a 
 set of decomposition numbers and the Maslov decomposition polynomial
is a set of polynomials.

In our explorations, all these decomposition invariants are quite different
than the invariants we have found through our use of generating families.  As an
illustration, we will just mention the decomposition invariants for the twist links.

\begin{proposition}  Let $L_{j,k}$ be the twist link; see \fullref{fig:Ljk}.  Then 
\begin{enumerate}
\item for any $j, k$, the decomposition number of $L_{j, k}$ is $4$, and 
 the decomposition polynomial of $L_{j, k}$ is $z^{-2} + 2 z^0 + z^2$;
\item if $j=k$, the set of Maslov decomposition numbers of $L_{j, k}$  is $\{1, 4\}$, 
and the set of Maslov decomposition polynomials  equals $\{ z^{2}, z^{-2} + 2z^0 + z^2 \}$;
if $j \neq k$, the set of Maslov decomposition numbers of $L_{j, k}$  is $\{1, 2\}$, and
 the set of Maslov decomposition polynomials equals $\{ z^{2}, z^{0} + z^{2}  \}$.
\end{enumerate}
\end{proposition}

\begin{proof}
First let us examine the non-Maslov decomposition numbers.  Consider the
front projection of $L_{j,k}$ as given by \fullref{fig:Ljk}.  For any $j, k$,
 there are $4$ allowable decompositions: one must
resolve all the $j+k$ self-strand crossings, and then one has the choice to 
resolve none of the other crossings, the $2$ uppermost crossings, the $2$ lowermost crossings, or both
the $2$ uppermost and $2$ lowermost crossings.  
  For the particular front projection of $L_{j,k}$ as given by \fullref{fig:Ljk},  the number of left cusps is given by $j + k + 2$, and 
 thus we see that
\begin{align*}
D(L_{j,k}) &=  z^{(j+k+2) - (j+k+0)} + 2 z^{(j+k+2) - (j+k+2)} + z^{(j+k+2)- (j+k+4)} \\
&= 
z^{2} + 2 z^0 + z^{-2}.
\end{align*}
Now consider the Maslov versions.  All the self-strand crossings are Maslov.  When $j=k$, either none of the inter-strand crossings are Maslov, or all the crossings are Maslov.  With the first type of labels, there is only
one admissible decomposition.  With the second type of labels, there are the same 
$4$ as were seen above.  Thus when $j=k$, the set of Maslov decompositions
is $\{ 1, 4 \}$ and the Maslov decomposition polynomials are 
$$ D^\mu(L_{j,k}) = \{  z^{2},  z^{2} + 2z^0  + z^{-2} \}.$$
When  $j\neq k$, there are essentially three ways to
label the branches:  either none of the inter-strand crossings are
Maslov, or only the top two inter-strand crossings are Maslov, or only
the bottom two inter-strand crossings are Maslov.  With the first
type of labels, we have one admissible decomposition.  With either the
second or third type of labels, we have two admissible decompositions. 
Hence the set of Maslov decomposition numbers is $\left\{ 1,2\right\} $.  The
set of associated Maslov decomposition polynomials is
$$ D^\mu(L_{j,k}) = \{  z^{2}, z^{2} + z^{0}\}.$$ 
See \fullref{fig:decomps}.
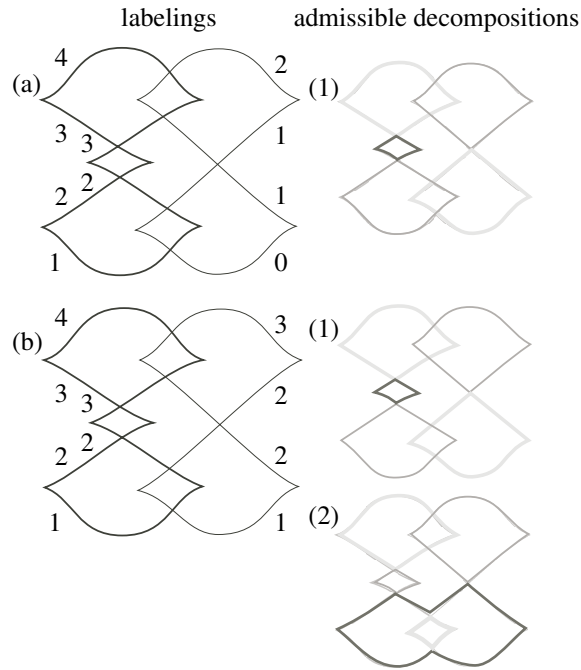
\begin{figure}[ht]
\labellist\small
\pinlabel {labelings} [t] at 175 750
\pinlabel {(a)} [l] at -10 655
\pinlabel {$4$} [br] at 68 671
\pinlabel {$3$} [tr] at 68 618
\pinlabel {$3$} [br] at 97 574
\pinlabel {$2$} [tr] at 97 563
\pinlabel {$2$} [br] at 68 515
\pinlabel {$1$} [tr] at 61 473
\pinlabel {$2$} [bl] at 287 663
\pinlabel {$1$} [tl] at 287 616
\pinlabel {$1$} [bl] at 287 516
\pinlabel {$0$} [tl] at 287 473
\pinlabel {(b)} [l] at -10 365
\pinlabel {$4$} [br] at 68 379
\pinlabel {$3$} [tr] at 68 326
\pinlabel {$3$} [br] at 97 282
\pinlabel {$2$} [tr] at 97 271
\pinlabel {$2$} [br] at 68 223
\pinlabel {$1$} [tr] at 61 181
\pinlabel {$3$} [bl] at 287 371
\pinlabel {$2$} [tl] at 287 324
\pinlabel {$2$} [bl] at 287 224
\pinlabel {$1$} [tl] at 287 181
\pinlabel {admissible decompositions} [t] at 475 750
\pinlabel {(1)} [r] at 375 650
\pinlabel {(1)} [r] at 375 380
\pinlabel {(2)} [r] at 375 170
\endlabellist
\centerline{\includegraphics[height=3.5in]{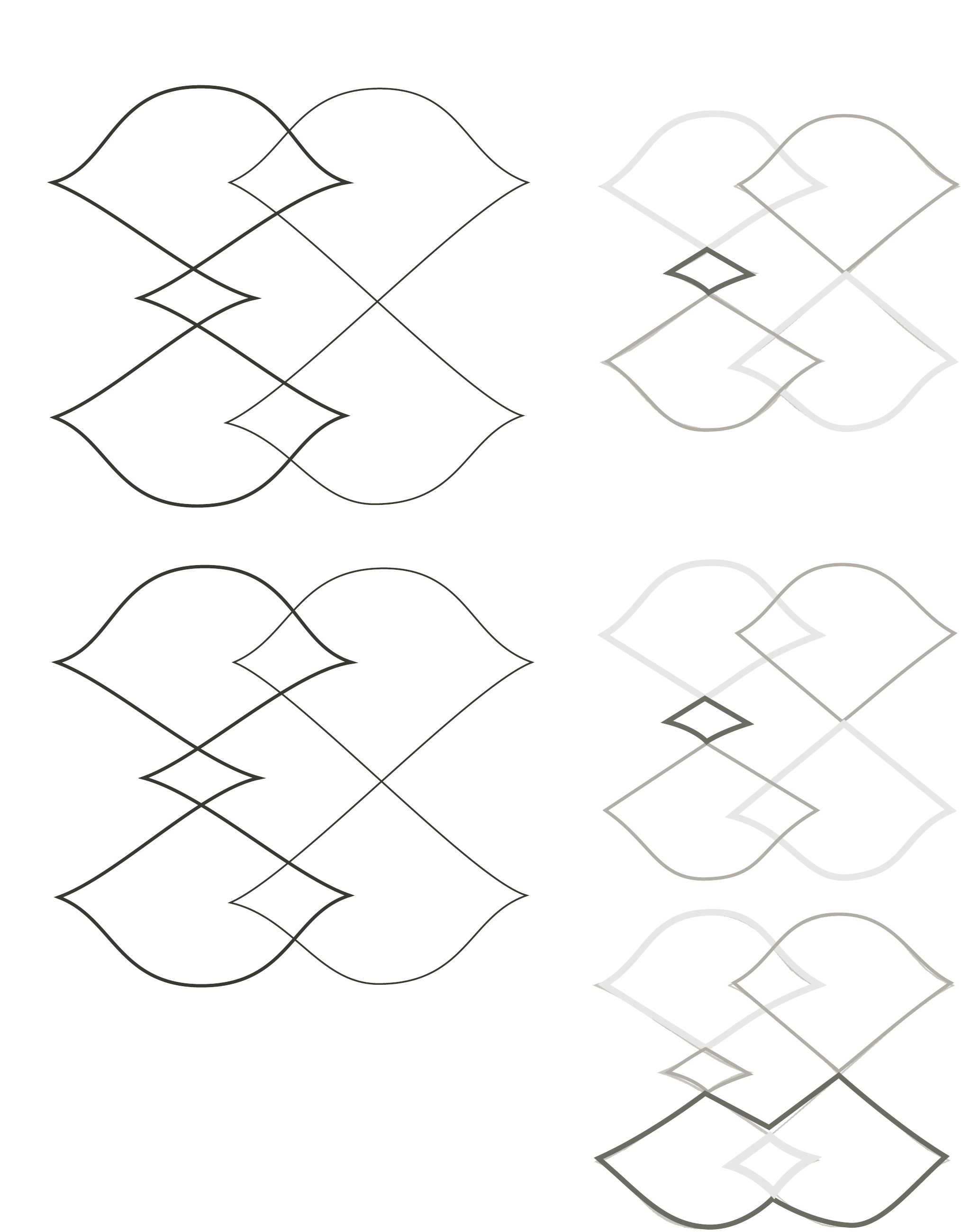}}
\caption{In (a), the branches of $L_{2,1}$ are labeled so that no
inter-strand crossings are Maslov.  The unique Maslov admissible decomposition
is shown.  In (b), the branches are labeled so that only the bottom set of
inter-strand crossings are Maslov.  The two Maslov admissible decompositions
are shown.  Note that a third labeling is possible, in particular one in
which only the top set of inter-strand crossings are Maslov.  The two
admissible decompositions in that case are symmetric to those in the second
labeling.}
        \label{fig:decomps}
\end{figure}
\end{proof}

In particular, we see that these decomposition invariants cannot distinguish, for example, 
the topologically equivalent Legendrian links $L_{1,4}$ and $L_{2,3}$ which can be distinguished by the generating family polynomials.  It would be interesting to see if the decomposition invariants
can be further refined so that they do capture the same information as the generating
family polynomials.

\subsection{DGA invariants} First we will briefly outline the process by
which a differential, graded
algebra (DGA) is associated to a link, and a polynomial is in turn associated
to the algebra.  This set-up is analogous to the work of Ng \cite{Ng1}.

The first step is to associate an algebra to the front projection of a link.
 The algebra $\mathcal{A}$ is given to be the free, unital, associative
algebra with coefficients in $\Z/2\Z$ generated by the
crossings and right cusps of the front projection of the link.  Each
generator is assigned a degree, which extends to all of $\mathcal{A}$ by
multiplicativity.  The degree of each cusp is $1$, while the degree of each
crossing is calculated using a simple procedure described in \cite[Section 2.2]{Ng1}.
 The degree depends upon a choice of marked points on
each component of the link:  a change in the choice of marked points may
shift some of the degrees by an integer.  In Ng and Traynor \cite{Ng-Traynor}, there was a
canonical choice for the marked points, but not for the links in $\R^3$ that
we are considering. 

The second step is to calculate a map $\partial \co \mathcal{A}\rightarrow 
\mathcal{A}$, called a differential, having the properties that $\partial
^{2}=0$ and $\partial $ lowers degree by one.
 $\partial $ is defined over $\Z/2\Z$ on the generators of $\mathcal{A}$ by counting certain immersed disks in the front projection. 
For details, see \cite[Section 2.2]{Ng1}.

In order to use the DGA to distinguish links, we use additional structure
properties described in \cite[2.4]{Ng-Traynor}, which is essentially 
Mishachev's relative homotopy splitting from \cite{Mishachev}. 
We break up $\mathcal{A}$ into components $\mathcal{A}^{1,1},
\mathcal{A}^{0,0},\mathcal{A}^{1,0},$ and $\mathcal{A}^{0,1}$, where the
idea is to separate the generators based on how strands from the link
components $\Lambda _{1}$ and $\Lambda _{0}$ contribute to each generator. 
For example, $\mathcal{A}^{1,0}$ is the module over $\Z/2\Z$
generated by words of the form $a_{i_{1}}\cdots a_{i_{m}}$ where the
overstrand of $a_{i_{1}}$ is in $\Lambda _{1}$ and the understrand of $a_{i_{m}}$ is in $\Lambda _{0}$, and the understrand of $a_{i_{p}}$ is in
the same link component as the overstrand of $a_{i_{p+1}}$.  $\mathcal{A}^{1,1}$ ($\mathcal{A}^{0,0}$) is generated by analogous words together with
an indeterminate $e_{1}$ ($e_{0}$).  The map $\partial $ is altered to be
defined on $\mathcal{A}^{i,j}$, and we call the new differential $\partial
^{\prime}$.  On $\mathcal{A}^{1,0}$ and $\mathcal{A}^{0,1}$, $\partial
^{\prime}$ agrees with $\partial $, while on $\mathcal{A}^{1,1}$ ($\mathcal{A}^{0,0}$) there is the slight modification that any $1$ term is replaced by 
$e_{1}$ ($e_{0}$).  Thus we have a link DGA given by $\left( \mathcal{A}^{\ast ,\ast},\partial ^{\prime}\right) $, and explicitly defined in 
\cite[Definition 2.13]{Ng-Traynor}.

Once we have a link DGA for a link $L$, we may use it to calculate a
polynomial invariant of $L$, up to equivalence.  A full description of this
process is given in \cite[Section 2.2]{Ng1} or  \cite[2.5]{Ng-Traynor}.  It involves defining an
augmentation $\epsilon $, which is a map from $\mathcal{A}$ to $\Z/2
\Z$.  Using this augmentation, we calculate $\partial _{\epsilon
}^{\prime}\co V_{k}^{i,j}\rightarrow V_{k-1}^{i,j}$, where $V_{k}^{i,j}$ is
the graded vector space over $\Z/2\Z$ generated by the
generators of $\mathcal{A}^{i,j}$ of degree $k$.  In particular, $V_{k}^{1,0}$ is generated by degree $k$ crossings with an overstrand in $\Lambda _{1}$ and an understrand in $\Lambda _{0}$.  Define 
\begin{equation*}
\beta _{k}^{1,0}\left( \epsilon \right) =\text{dim}\frac{\ker \partial
_{\epsilon}^{\prime}\co V_{k}^{1,0}\rightarrow V_{k-1}^{1,0}}{\text{im}
\partial _{\epsilon}^{\prime}\co V_{k+1}^{1,0}\rightarrow V_{k}^{1,0}},
\end{equation*}
and define 
\begin{equation*}
\chi _{\epsilon}^{1,0}\left( \lambda \right) \left[ L\right] =\sum_{k}\beta
_{k}^{1,0}\left( \epsilon \right) \cdot \lambda ^{k}\text{.}
\end{equation*}
Then $\chi _{\epsilon}^{1,0}$ is one of the \emph{split Poincar\'{e}--Chekanov polynomials of} $L$ \emph{with respect to} $\epsilon $. 
Polynomials $\chi _{\epsilon}^{1,1}$, $\chi _{\epsilon}^{0,0}$, and $\chi
_{\epsilon}^{0,1}$ can be defined similarly.  The polynomials
$\chi _{\epsilon}^{1,0}$ and $\chi_{\epsilon}^{0,1}$
depend on the grading of the link; in particular, changing the choice of
marked points on the components has
the potential to shift the degree of the polynomials up or down by some
integer.  We will focus on normalized versions of the $\chi _{\epsilon
}^{1,0}$ polynomials.  Therefore we will define the \emph{normalized
negative Poincar\'{e}--Chekanov polynomial} as
\begin{equation*}
\chi _{\epsilon}^{-}=\chi _{\epsilon}^{1,0}\cdot \lambda ^{m},
\end{equation*}
where $m$ is chosen so that $\chi _{\epsilon}^{-}$ has degree $0$.  The
set $\left\{ \chi _{\epsilon}^{-}\right\} $ arising from different choices
of $\epsilon $ is an invariant of a Legendrian link $\left( \Lambda
_{1},\Lambda _{0}\right) $.  In the cases studied in this paper, there will
always be a unique polynomial and so the set notation will be dropped.

We first compute the normalized negative Poincar\'e-Chekanov polynomial of a
Legendrian rational link; compare \fullref{theorem 6.1}.

\begin{theorem}
\label{PCpolyratlink}Let $L=\left( 2w_{n},k_{n},\dots
,2w_{1},k_{1},2w_{0}\right) $ (see \fullref{fig:rational_link}).  Then 
\begin{equation*}
\chi _{\epsilon}^{-}\left( \lambda \right) \left[ L\right] =w_{0}\lambda
^{0}+w_{1}\lambda ^{-k_{1}}+w_{2}\lambda ^{-\left( k_{1}+k_{2}\right)
}+\dots +w_{n}\lambda ^{-\left( k_{1}+k_{2}+\dots +k_{n}\right)}\text{.}
\end{equation*}
\end{theorem}

\begin{proof}
We use $c_{i}$'s and $v_{i}$'s to denote generators arising from right cusps
and vertical crossings, respectively.  Let $h_{1},h_{2},\dots ,h_{2\sum
w_{i}}$ be the $2\sum_{i=0}^{n}w_{i}$ generators arising from the horizontal
crossings, numbered right to left. For example, \fullref{fig:dga_rat} shows
the generators for $L=\left( 2,1,4,2,2\right) $.

\begin{figure}[ht]
\labellist\small
\pinlabel {$\Lambda_1$} [br] at 55 221
\pinlabel {$\Lambda_0$} [tr] at 55 75
\pinlabel {$h_1$} [b] at 465 224
\pinlabel {$h_2$} [b] at 404 224
\pinlabel {$h_3$} [b] at 310 175
\pinlabel {$h_4$} [b] at 275 175
\pinlabel {$h_5$} [b] at 240 175
\pinlabel {$h_6$} [b] at 205 175
\pinlabel {$h_7$} [b] at 118 146
\pinlabel {$h_8$} [b] at 72 150
\pinlabel {$c_1$} [l] at 513 253
\pinlabel {$c_2$} [l] at 508 199
\pinlabel {$c_3$} [l] at 508 144
\pinlabel {$c_4$} [l] at 295 87
\pinlabel {$c_5$} [l] at 295 38
\pinlabel {$v_1$} [b] at 437 170
\pinlabel {$v_2$} [b] at 260 100
\pinlabel {$v_3$} [t] at 252 66
\endlabellist
\centerline{\includegraphics[height=2.5in]{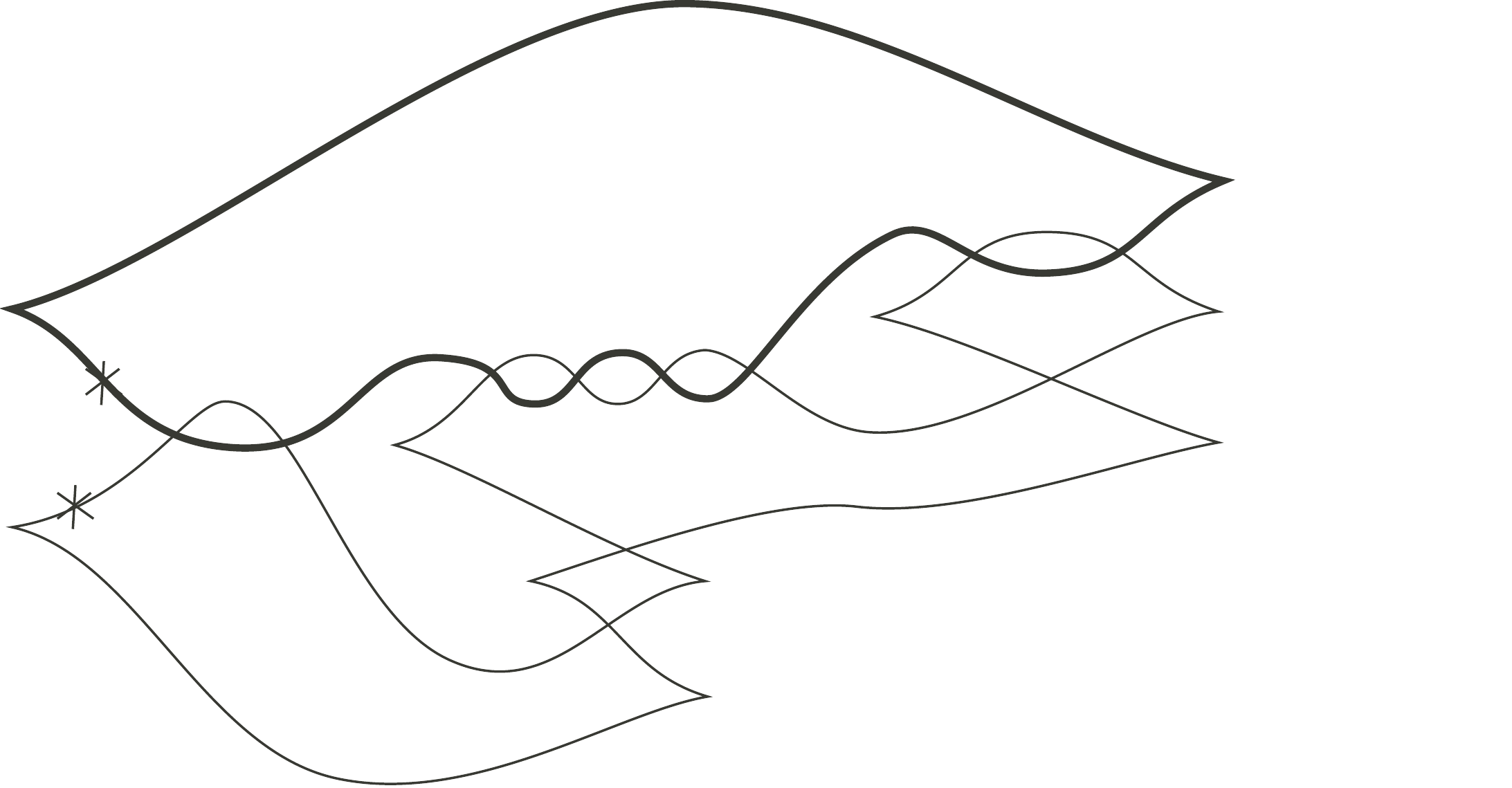}}
\caption{The generators of $\mathcal{A}$ for $L=\left(
2,1,4,2,2\right) $ are labeled using $c_{i}$'s, $v_{i}$'s, and $h_{i}$'s,
and the $\ast $'s on each component of $L$ indicate the choice of marked
points.}
        \label{fig:dga_rat}
\end{figure}
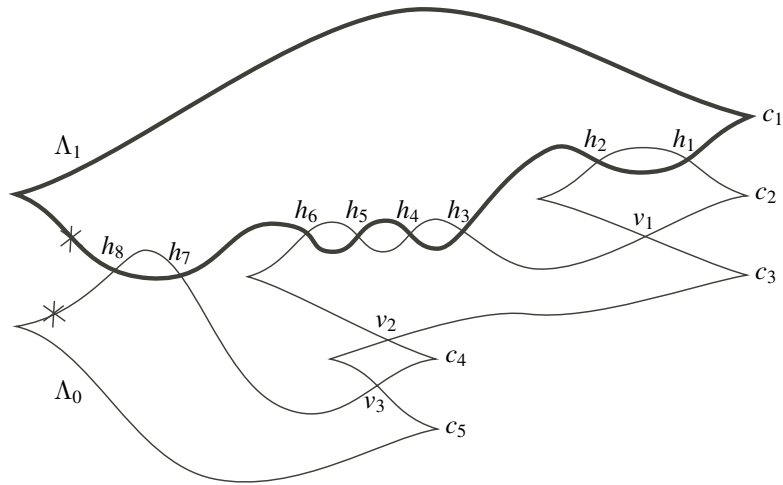

Since we are calculating $\chi _{\epsilon}^{-}$, we focus on generators in $V^{1,0}$, namely
$\left\{ h_{m}\mid m\text{ is even}\right\}$.
We can
choose marked points so that the degrees of the $h_{m}$ are as follows:
\begin{equation*}
\text{deg}h_{m}=\left\{ 
\begin{array}{lll}
\bigl( \sum_{i=1}^{n}k_{i}\bigr) \left( -1\right) ^{m+1} & \text{if} & 
1\leq m\leq 2w_{n} \\ 
\bigl( \sum_{i=1}^{n-1}k_{i}\bigr) \left( -1\right) ^{m+1} & \text{if} & 
2w_{n}+1\leq m\leq 2\left( w_{n}+w_{n-1}\right) \\ 
\qquad \quad \vdots &  & \qquad \qquad \vdots \\ 
\left( k_{1}\right) \left( -1\right) ^{m+1} & \text{if} & 2\left(
\sum_{i=2}^{n}w_{i}\right) +1\leq m\leq 2\left( \sum_{i=1}^{n}w_{i}\right)
\\ 
0 & \text{if} & 2\left( \sum_{i=1}^{n}w_{i}\right) +1\leq m\leq 2\left(
\sum_{i=0}^{n}w_{i}\right)
\end{array}
\right.
\end{equation*}
For example, in the link $L=\left( 2,1,4,2,2\right) $ with marked points
as indicated in \fullref{fig:dga_rat}, we have the following:
\begin{equation*}
\text{deg}h_{m}=\left\{ 
\begin{array}{lll}
3 & \text{if} & m=1 \\ 
-3 & \text{if} & m=2 \\ 
2 & \text{if} & m=3,5 \\ 
-2 & \text{if} & m=4,6 \\ 
0 & \text{if} & m=7,8.
\end{array}
\right.
\end{equation*}
This link has a unique augmentation $\epsilon $, for which $\epsilon \left(
h_{m}\right) =0$ for all $m$. Each $h_{m}$ is in the kernel of the $\partial _{\epsilon}^{\prime}$ map of the appropriate degree, but no $h_{m}$s are in the image of any $\partial _{\epsilon}^{\prime}$
map. Thus we have
\begin{equation*}
\chi _{\epsilon}^{1,0}\left( \lambda \right) \left[ L\right]
=\sum_{j=-\infty}^{\infty}N\left( h_{m},j\right) \cdot \lambda ^{j}\text{,}
\end{equation*}
where $N\left( h_{m},j\right) $ is the number of generators $h_{m}\in 
\mathcal{A}^{1,0}$ of degree $j$. Hence, checking the degrees of the $h_{m} $ given above, we have
\begin{equation*}
\chi _{\epsilon}^{1,0}\left( \lambda \right) \left[ L\right] =w_{0}\lambda
^{0}+w_{1}\lambda ^{-\left( k_{1}\right)}+w_{2}\lambda ^{-\left(
k_{1}+k_{2}\right)}+\dots +w_{n}\lambda ^{-\left( k_{1}+k_{2}+\dots
+k_{n}\right)}\text{.}
\end{equation*}
Since this polynomial is of degree $0$, it is in fact equal to $\chi
_{\epsilon}^{-}$.\hfill
\end{proof}

Lastly, we compute the normalized negative Poincar\'e--Chekanov polynomial of a
Legendrian twist link; compare \fullref{Ljkpolys}.

\begin{theorem}
\label{PCpolytwlink}Let $L_{j,k}$ be the twist link described in \fullref{fig:Ljk}. Then
\begin{equation*}
\chi _{\epsilon}^{-}\left( \lambda \right) [L_{j,k}] =\lambda
^{0}+\lambda ^{-\left| j-k\right|}.
\end{equation*}
\end{theorem}

\begin{proof}
Use $c_{i}$'s and $s_{i}$'s to denote generators arising from the right
cusps and self-strand crossings of $L_{j,k}$, respectively.  Use $x_{1},x_{2},x_{3}$, and $x_{4}$ to denote generators arising from the four
interstrand crossings, numbered from top to bottom.  See \fullref{fig:dga_twist} for an example of how to label generators.

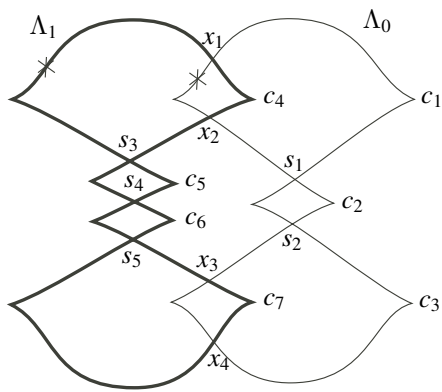
\begin{figure}[ht]
\labellist\small
\pinlabel {$\Lambda_1$} [br] at 51 310
\pinlabel {$\Lambda_0$} [bl] at 313 315
\pinlabel {$c_1$} [l] at 363 261
\pinlabel {$c_2$} [l] at 292 167
\pinlabel {$c_3$} [l] at 361 78
\pinlabel {$c_4$} [l] at 222 261
\pinlabel {$c_5$} [l] at 152 186
\pinlabel {$c_6$} [l] at 152 152
\pinlabel {$c_7$} [l] at 222 80
\pinlabel {$s_1$} [b] at 257 189
\pinlabel {$s_2$} [t] at 255 147
\pinlabel {$s_3$} [b] at 109 207
\pinlabel {$s_4$} [b] at 114 171
\pinlabel {$s_5$} [t] at 112 132
\pinlabel {$x_1$} [b] at 185 299
\pinlabel {$x_2$} [t] at 181 243
\pinlabel {$x_3$} [b] at 180 100
\pinlabel {$x_4$} [t] at 190 40
\endlabellist
\centerline{\includegraphics[height=2in]{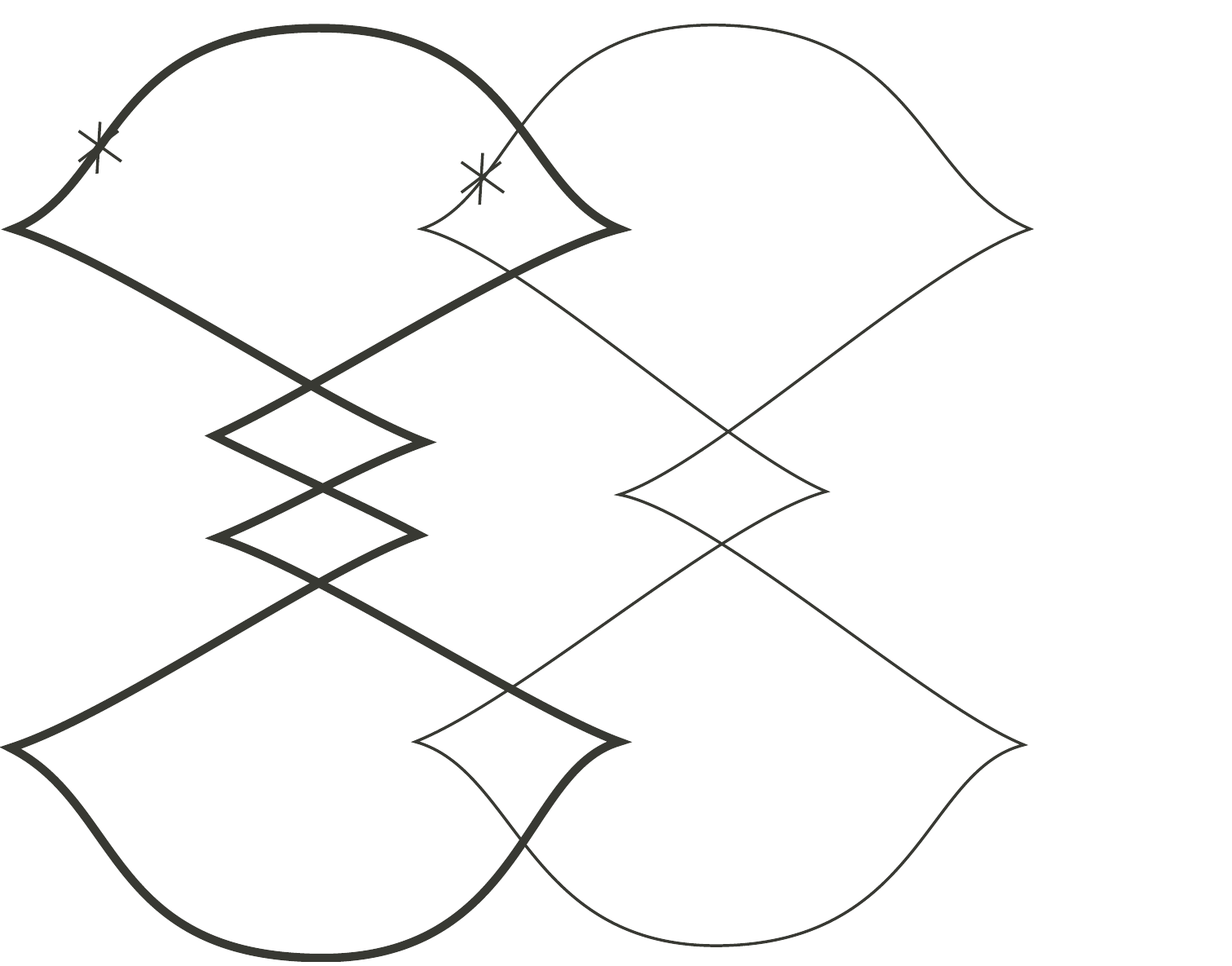}}
\caption{Labeling of generators of $\mathcal{A}$ for
the link $L_{3,2}$, where the $\ast $'s on each component of $L_{3,2}$
indicate the choice of marked points}
        \label{fig:dga_twist}
\end{figure}
Note that $\mathcal{A}^{1,0}$ is generated by $\left\{ x_{1},x_{3}\right\} $.  We can choose marked points so that deg$\left( x_{1}\right) =0$ and
deg$\left( x_{3}\right) = k - j$.
We have a unique augmentation $\epsilon $.  Calculating $\partial
_{\epsilon}^{\prime}$, we find that both $x_{1}$ and $x_{3}$ are in the
kernel of the $\partial _{\epsilon}^{\prime}$ map of the appropriate
degree, and neither is in the image of a $\partial _{\epsilon}^{\prime}$
map.  Thus we have 
\begin{equation*}
\chi _{\epsilon}^{1,0}\left( \lambda \right) [L_{j,k}] =\lambda
^{0}+\lambda ^{ k-j}\text{.}
\end{equation*}
The formula for the normalized Poincar\'e--Chekanov polynomial follows.
\end{proof}

Summarizing the results from Theorems  \ref{theorem 6.1}, 
\ref{PCpolyratlink}, \ref{Ljkpolys}, and \ref{PCpolytwlink}, we find: 
\begin{theorem} \label{same_polys}
Let $L$ be a Legendrian link that is either a rational link or a twist
link.  Then the negative homology polynomial of $L$ is the same as its
normalized negative Poincar\'{e}--Chekanov polynomial.
\end{theorem}

\bibliographystyle{gtart}
\bibliography{link}

\begin{thebibliography}{}
\providecommand\bibmarginpar{\leavevmode\marginpar}
\def\urlstyle#1{{\tt #1}}

\bibitem{Adams}
\textbf{C\,C Adams}, \emph{The knot book}, W\,H Freeman and Company, New York
  (1994) \xox{MR}{1266837}

\bibitem{Arnold}
\textbf{V\,I Arnol'd}, \textbf{S\,M Guse{\u\i}n-Zade}, \textbf{A\,N Varchenko},
  \emph{Singularities of differentiable maps. {V}ol. {I}}, Monographs in
  Mathematics 82, Birkh\"auser, Boston (1985) \xox{MR}{777682}

\bibitem{Chekanov}
\textbf{Y Chekanov}, \href{http://dx.doi.org/10.1007/s002220200212}
  {\emph{Differential algebra of {L}egendrian links}}, Invent. Math. 150 (2002)
  441--483 \xox{MR}{1946550}

\bibitem{Chekanov2}
\textbf{Y\,V Chekanov}, \emph{Invariants of {L}egendrian knots}, from:
  ``Proceedings of the International Congress of Mathematicians, Vol. II
  (Beijing, 2002)'', Higher Ed. Press, Beijing (2002)  385--394
  \xox{MR}{1957049}

\bibitem{Conway}
\textbf{J\,H Conway}, \emph{An enumeration of knots and links, and some of
  their algebraic properties}, from: ``Computational Problems in Abstract
  Algebra (Proc. Conf., Oxford, 1967)'', Pergamon, Oxford (1970)  329--358
  \xox{MR}{0258014}

\bibitem{EGH}
\textbf{Y Eliashberg}, \textbf{A Givental}, \textbf{H Hofer},
  \emph{Introduction to symplectic field theory}, Geom. Funct. Anal.  (2000)
  560--673 \xox{MR}{1826267}

\bibitem{EG}
\textbf{Y Eliashberg}, \textbf{M Gromov}, \emph{Lagrangian intersection theory:
  finite-dimensional approach}, from: ``Geometry of differential equations'',
  Amer. Math. Soc. Transl. Ser. 2 186, Amer. Math. Soc., Providence, RI (1998)
  27--118 \xox{MR}{1732407}

\bibitem{Ernst}
\textbf{C Ernst}, \textbf{D\,W Sumners}, \emph{A calculus for rational tangles:
  applications to {DNA} recombination}, Math. Proc. Cambridge Philos. Soc. 108
  (1990) 489--515 \xox{MR}{1068451}

\bibitem{Etnyre}
\textbf{J\,B Etnyre}, \emph{Legendrian and transversal knots}, from: ``Handbook
  of knot theory'', Elsevier B. V., Amsterdam (2005)  105--185
  \xox{MR}{2179261}

\bibitem{ENS}
\textbf{J\,B Etnyre}, \textbf{L\,L Ng}, \textbf{J\,M Sabloff},
  \href{http://projecteuclid.org/getRecord?id=euclid.jsg/1092316653}
  {\emph{Invariants of {L}egendrian knots and coherent orientations}}, J.
  Symplectic Geom. 1 (2002) 321--367 \xox{MR}{1959585}

\bibitem{Mishachev}
\textbf{K Mishachev},
  \href{http://projecteuclid.org/getRecord?id=euclid.jsg/1092749564} {\emph{The
  {$N$}-copy of a topologically trivial {L}egendrian knot}}, J. Symplectic
  Geom. 1 (2003) 659--682 \xox{MR}{2039159}

\bibitem{Ng1}
\textbf{L\,L Ng}, \href{http://dx.doi.org/10.1016/S0040-9383(02)00010-1}
  {\emph{Computable {L}egendrian invariants}}, Topology 42 (2003) 55--82
  \xox{MR}{1928645}

\bibitem{Ng-Traynor}
\textbf{L Ng}, \textbf{L Traynor},
  \href{http://projecteuclid.org/getRecord?id=euclid.jsg/1118755327}
  {\emph{Legendrian solid-torus links}}, J. Symplectic Geom. 2 (2004) 411--443
  \xox{MR}{2131643}

\bibitem{Chekanov+Pushkar}
\textbf{P\,E Pushkar'}, \textbf{Y\,V Chekanov}, \emph{Combinatorics of fronts
  of {L}egendrian links, and {A}rnol'd's 4-conjectures}, Uspekhi Mat. Nauk 60
  (2005) 99--154 \xox{MR}{2145660}

\bibitem{Theret}
\textbf{D Th{\'e}ret}, \href{http://dx.doi.org/10.1016/S0166-8641(98)00049-2}
  {\emph{A complete proof of {V}iterbo's uniqueness theorem on generating
  functions}}, Topology Appl. 96 (1999) 249--266 \xox{MR}{1709692}

\bibitem{TheretCamel}
\textbf{D Th{\'e}ret}, \href{http://dx.doi.org/10.1007/s000140050107} {\emph{A
  {L}agrangian camel}}, Comment. Math. Helv. 74 (1999) 591--614
  \xox{MR}{1730659}

\bibitem{Symplectic}
\textbf{L Traynor}, \href{http://dx.doi.org/10.1007/BF01896659}
  {\emph{Symplectic homology via generating functions}}, Geom. Funct. Anal. 4
  (1994) 718--748 \xox{MR}{1302337}

\bibitem{CircularHelix}
\textbf{L Traynor}, \href{http://dx.doi.org/10.1017/S030500419700193X}
  {\emph{Legendrian circular helix links}}, Math. Proc. Cambridge Philos. Soc.
  122 (1997) 301--314 \xox{MR}{1458235}

\bibitem{Generating}
\textbf{L Traynor}, \href{http://dx.doi.org/10.2140/gt.2001.5.719}
  {\emph{Generating function polynomials for {L}egendrian links}}, Geom. Topol.
  5 (2001) 719--760 \xox{MR}{1871403}

\bibitem{Viterbo}
\textbf{C Viterbo}, \href{http://dx.doi.org/10.1007/BF01444643}
  {\emph{Symplectic topology as the geometry of generating functions}}, Math.
  Ann. 292 (1992) 685--710 \xox{MR}{1157321}

\end{thebibliography}

\end{document}